\documentclass[11pt]{article}
\usepackage{amsmath,amsthm,amsfonts,epsfig}
\usepackage{amssymb}
\usepackage{enumerate}
\usepackage[mathscr]{euscript}
\usepackage{tikz}

\newtheorem{thm}{Theorem}[section]
\newtheorem{lemma}[thm]{Lemma}
\newtheorem{cor}[thm]{Corollary}
\theoremstyle{definition}
\newtheorem{defn}[thm]{Definition}
\newtheorem{nota}[thm]{Notation}

\newtheorem{remark}[thm]{Remark}
\newtheorem{prop}[thm]{Proposition}
\newtheorem{ex}[thm]{Example}

\newcommand{\N}{\mathbb{N}}
\newcommand{\Z}{\mathbb{Z}}

\newcommand{\R}{\mathbb{R}}
\newcommand{\C}{\mathbb{C}}

\newcommand{\ms}{\mathscr}
\newcommand{\pec}[1]{\overline{#1}^P}

\newcommand{\eps}{\varepsilon}

\DeclareMathOperator{\dM}{dist_M}

\DeclareMathOperator{\Mod}{\text{Mod}}
\DeclareMathOperator{\dist}{dist}
\DeclareMathOperator{\diam}{diam}
\DeclareMathOperator{\dum}{d_M}
\DeclareMathOperator{\dumm}{d_{M}^\prime}

\title{Carath\'eodory-type extension theorem with respect to prime end boundaries}
\author{Joshua Kline, Jeff Lindquist and Nageswari Shanmugalingam
\footnote{The third author was partially supported by grant DMS~\#1800161 from NSF (U.S.A.)}}

\begin{document}
\maketitle

\begin{abstract} 
We prove a Caratheodory-type extension of BQS homeomorphisms between two domains in proper, locally path-connected metric spaces as homeomorphisms between their prime end closures. We also give a Caratheodory-type extension of geometric quasiconformal mappings between two such domains provided the two domains are both Ahlfors $Q$-regular and support a $Q$-Poincare inequality when equipped with their respective Mazurkiewicz metrics. We also provide examples to demonstrate the strengths and weaknesses of prime end closures in this context.
\end{abstract}

\section{Introduction}

The celebrated theorem of Riemann states that every simply connected planar domain that is not the entire
complex plane is conformally equivalent to the unit disk (that is, there is a conformal mapping between
the domain and the disk). However, there are plenty of simply connected planar domains with quite complicated
boundaries, and it is not always possible to extend these conformal mappings from such domains to
their (topological) boundaries. The work of Carath\'eodory~\cite{Car} 
beautifully overcame this obstruction by the use
of cross-cuts and the corresponding notion of prime ends, 
and showed that every such conformal mapping extends to a homeomorphism
between the (Carath\'eodory) prime end boundary
of the domain and the topological boundary (the unit circle) of the
disk. This extension theorem is known in various literature as 
the Carath\'eodory theorem for conformal mappings 
and as the prime end theorem (see~\cite{Car}, \cite[Section~4.6]{Ahl} and~\cite[Chapter~2]{Pom}).

In the case of non-simply connected planar domains or domains in higher dimensional 
Euclidean spaces (and 
Riemannian manifolds), the notion of prime ends of Carath\'eodory is not as fruitful. Indeed, it is not clear
what should play the role of cross-cuts, as not all Jordan arcs in the domain
with both end points in the boundary of the domain will separate the domain into exactly two connected 
components. There are viable extensions of the Carath\'eodory construction for certain types of 
Euclidean domains, called quasiconformally collared domains (see for example~\cite{Na, AW}),
but this construction is not optimal when the domain of interest is not quasiconformally collared,
and indeed, there is currently no analog of quasiconformal collared domains in the setting of
general metric measure spaces.
However, the workhorse of the prime end construction of Carath\'eodory is the nested sequence
of components of the compliments of the Jordan arcs that make up the end. Keeping this in mind, an
alternate notion of prime ends was proposed in~\cite{ABBS}. The construction given there 
proves to be useful even
in the non-smooth setting of domains in metric measure spaces, and is a productive tool in potential
theory and the Dirichlet problem for such domains, see~\cite{BBS2, ES, AS}. 

Branched quasisymmetric mappings (BQS) 
between metric spaces started life as extensions of quasiregular mappings
between manifolds, see~\cite{Guo1, Guo2, GW1, GW2, LP}. The work of Guo and Williams~\cite{GW2} demonstrate the utility of BQS mappings in Stoilow-type factorizations of quasiconformal maps between
metric measure spaces. Between metric spaces of bounded turning, the class of homeomorphisms that are
BQS coincide with the class of quasisymmetric maps. However, for more general metric spaces,
BQS homeomorphisms are not necessarily quasisymmetric, as demonstrated by the slit disk (see
the discussion in Section~2). 
The focus of this present paper
is to obtain a Carath\'eodory type extension theorem for two types of mappings between domains in locally
path-connected metric spaces using the construction of prime ends from~\cite{ABBS}. The two types of 
mappings considered are \emph{branched quasisymmetrc} (BQS) homeomorphisms and \emph{geometric quasiconformal}
homeomorphisms. In Theorem~\ref{thm:BQS-fp-homeo} we prove 
that every BQS homeomorphism between two bounded domains in locally path-connected 
metric spaces has a homeomorphic extension to the respective prime end closures. We also show that if
these domains are finitely connected at their respective 
boundaries, then the extended homeomorphism is also BQS, see Proposition~\ref{prop:BQS-Mazur} and 
the subsequent remark. 
In Theorem~\ref{thm:BQS-unbdd} we also prove a Carath\'eodory extension theorem for BQS mappings
between unbounded domains in locally path-connected metric spaces, where the notion of prime ends
for unbounded domains is an extension of the one from~\cite{ABBS}, first formulated in~\cite{E}.
We point out here that local path-connectedness of the ambient metric spaces
can be weakened to local continuum-connectedness (that is, for each point in the metric space and
each neighborhood of that point, there is a smaller neighborhood of the point such that each pair
of points in that smaller neighborhood can be connected by a continuum in that smaller neighborhood).
However, as local path-connectedness is a more familiar property, we will phrase our requirements in
terms of it. 

In contrast to BQS homeomorphisms, geometric quasiconformal mappings do not always have a homeomorphic extension
to the prime end closure for the prime ends as constructed in~\cite{ABBS}; 
for example, with $\Omega$ the planar unit disk and $\Omega^\prime$ the harmonic comb given by
\begin{equation}\label{eq:example-comb}
\Omega^\prime:=(0,1)\times(0,1)\setminus\left(\bigcup_n\{\tfrac{1}{n+1}\}\times[0,\tfrac12]\right),
\end{equation}
we know that $\Omega$ and $\Omega^\prime$
are conformally equivalent, but the prime end closure $\overline{\Omega}^P$
of $\Omega$ is the closed unit disk which is compact, while
$\overline{\Omega^\prime}^P$ is not even sequentially compact.
However, we show that if both domains are Ahlfors $Q$-regular Loewner
domains for some $Q>1$ when equipped with the induced Mazurkiewicz metrics $\dum$ and $\dumm$ 
respectively, then geometric quasiconformal mappings
between them extend to homeomorphisms between their prime end closures. Should the two domains 
be bounded, then the results in this paper also imply that the geometric quasiconformal mappings
extend as homeomorphisms between the respective Mazurkiewicz boundaries, with the extended map
serving as a quasisymmetric map thanks to~\cite[Theorem~4.9]{HeiK} together with~\cite[Theorem~3.3]{Korte}.
In Theorem~\ref{thm:main-geomQC} we give a Carath\'eodory-type extension for geometric quasiconformal maps.

Our results are related to those of~\cite{A, AW} where extensions, to the prime end closures, 
of certain classes of 
homeomorphisms between domains is demonstrated. In~\cite[Theorem~5]{A} it is shown that a 
homeomorphism between two bounded domains, one of which is of bounded turning and the other of which
is finitely connected at the boundary, has a homeomorphic extension to the relevant prime end 
closures provided that the homeomorphism does not map two connected sets that are a positive distance
apart to two sets that are zero distance apart (i.e., pinched). 
In general, it is not easy to verify that a
given BQS or geometric quasiconformal map satisfies this latter non-pinching condition. 
Moreover, we do not assume that the domains are finitely connected at the boundary. We therefore
prove homeomorphic extension results by hand. The paper~\cite{AW} deals with domains in the
first Heisenberg group $\mathbf{H}^1$, and the prime ends considered there are \emph{not} those
considered here. The notion of prime ends considered in~\cite{AW} are effectively that of Carath\'eodory,
see~\cite[Definition~3.3]{AW}. In particular, topologically three-dimensional analogs of 
the harmonic comb, e.g.~ $\Omega^\prime\times(0,1)$ with $\Omega^\prime$ as 
in~\eqref{eq:example-comb}, will have no prime end considered in~\cite{AW}
with impression containing any point in
$\{0\}\times\{\tfrac12\}\times[0,1]$, whereas each point in this set forms an impression of a
singleton prime end with respect to the notion of prime ends considered here.
The paper~\cite{Sev} has a similar Carath\'eodory type extension results for the so-called
ring mappings between two domains satisfying certain geometric restrictions at their respective
boundaries.

The structure of the paper is as follows. We give the background notions and results related to prime
ends and BQS/geometric quasiconformal maps in Section~2. In Section~3 we describe 
bounded geometry and the Loewner property; these notions are not needed in the subsequent
study of BQS mappings (Sections~4,5) but are needed in the study of geometric quasiconformal maps in
Section~6. Using the prime end boundary for bounded domains as defined in Section~2, in
Section~4 we prove a Carath\'eodory-type extension property of BQS homeomorphisms between bounded domains
in proper locally path-connected metric spaces. In Section~5 we extend the prime end construction to
unbounded domains, and demonstrate a Carath\'eodory-type extension of BQS homeomorphisms between
unbounded domains in proper locally path-connected metric spaces. The final section is concerned with
proving a Carath\'eodory-type extension property of geometric quasiconformal mappings between two domains
in proper locally path-connected metric spaces, under certain geometric restrictions placed on the 
two domains. Note that it is not possible to have a BQS homeomorphism between a bounded domain and an
unbounded domain, but it is possible to have a geometric quasiconformal map between a bounded domain
and an unbounded domain, as evidenced by the conformal map between the unit disk and the half-plane.

\section{Background notions}\label{sec:bdd-prime}

In this section we give a definition of prime ends for bounded domains. An extension of this definition of
prime ends (following~\cite{E})
for unbounded domain is postponed until Section~5 where it is first needed.
We follow \cite[Section 4]{ABBS} in constructing prime ends for bounded domains. 

We say that a metric space $X$ is (metrically) doubling if there is a positive integer $N$
such that whenever $x\in X$ and $r>0$, we can cover the ball $B(x,r)$ by at most $N$ number
of balls of radius $r/2$. 

Let $X$ be a complete, doubling metric space.  
It is known that a complete doubling space is proper, that is, closed and bounded subsets
of the space are compact.
Let $\Omega \subseteq X$ be a domain, that is, $\Omega$ is an open, non-empty,
connected subset of $X$.

\begin{defn}[Mazurkiewicz metric]\label{def:Mazurk}
The Mazurkiewicz metric $\dum$ on $\Omega$ is given by 
\[
\dum(x,y):=\inf_E\diam(E)
\]
for $x,y\in\Omega$, where the infimum is over all continua (compact connected sets) $E\subset\Omega$ that contain $x$ and $y$.
\end{defn}

Note that in general $\dum$ can take on the value of $\infty$; however, if $X$ is locally path-connected (and hence so is $\Omega$),
then $\dum$ is finite-valued and so is an actual metric on $\Omega$.

\begin{defn}[Acceptable Sets]
A bounded, connected subset $E \subseteq \Omega$ is an {\em acceptable set} if 
$\overline{E}$ is compact and 
$\overline{E} \cap \partial \Omega \neq \emptyset$.
\end{defn}

The requirement that $\overline{E}$ be compact was not needed to be explicitly stated in~\cite{ABBS} because it was 
assumed there that the metric space $X$ is complete and doubling, and so closed and bounded subsets are compact;
as $\Omega$ is bounded and $E\subset\Omega$, it follows then that $\overline{E}$ is compact in the setting
considered in~\cite{ABBS}. In our more general setting, we make this an explicit requirement.



Given two sets $E,F\subset\Omega$, we know from Definition~\ref{def:Mazurk} that
\[
\dM(E,F)=\inf_\gamma\diam(\gamma),
\]
where the infimum is over all continua $\gamma\subset\Omega$ with $\gamma\cap E\ne\emptyset\ne\gamma\cap F$.

\begin{defn}[Chains]\label{def:chains}
A sequence $\{E_i\}$ of acceptable sets is a {\em chain} is
\begin{enumerate}[(a)]
\item $E_{i+1} \subseteq E_i$ for all $i \geq 1$
\item $\dM(\Omega \cap \partial E_{i+1}, \Omega \cap \partial E_i) > 0$ for all $i \geq 1$
\item $\bigcap_{i=1}^{\infty} \overline{E_i} \subseteq \partial \Omega$. 
\end{enumerate}
\end{defn}

The second condition above guarantees also that $E_{i+1}$ is a subset of the interior $\text{int}(E_k)$ of
$E_k$.

\begin{defn}[Divisibility of Chains, and resulting ends]
A chain $\{E_i\}$ {\em divides} a chain $\{F_j\}$ (written $\{E_i\} | \{F_j\}$) if for each positive integer $k$ there exists
a positive integer $i_k$ with $E_{i_k} \subseteq F_k$.  Two chains $\{E_i\}$ and $\{F_j\}$ are {\em equivalent} (written 
$\{E_i\} \sim \{F_j\}$) if $\{E_i\} | \{F_j\}$ and $\{F_j\} | \{E_i\}$.  The equivalence class of a given chain 
$\{E_i\}$ under this relation is called an {\em end} and is denoted $[E_i]$.
\end{defn}

\begin{nota}
We write ends using the Euler script font (e.g. $\ms{E}, \ms{F}, \ms{G}$).  If $\{E_i\}$ is a chain representing the 
end $\ms{E}$, we write $\{E_i\} \in \ms{E}$.  
\end{nota}

Note that if $\{E_k\} | \{G_k\}$ and $\{F_k\}\sim\{E_k\}$, then $\{F_k\} | \{G_k\}$. 
Moreover, if $\{E_k\} | \{G_k\}$ and $\{F_k\}\sim\{G_k\}$, then $\{E_k\} | \{F_k\}$. 
Therefore the notion of divisibility is
inherited by ends from their constituent chains.

\begin{defn}[Divisibility of Ends]
An end $\ms{E}$ {\em divides} an end $\ms{F}$ (written $\ms{E} | \ms{F}$) if whenever $E_i \in \ms{E}$ and $F_j \in \ms{F}$, then $\{E_i\} | \{F_j\}$.
\end{defn}

\begin{defn}[Impressions]
The {\em impression} of a chain $\{E_i\}$ is $\bigcap_{i=1}^{\infty} \overline{E_i}$.  If two chains are equivalent, then they have the same impression.  Hence, if $\mathscr{E}$ is an end, we may refer to the impression $I(\mathscr{E})$ of that end.  We may also write $I[E_i]$ if $\{E_i\} \in \ms{E}$.    
\end{defn}

\begin{defn}[Prime Ends and Prime End Boundary]
An end $\ms{E}$ is called a {\em prime end} if $\ms{F}=\ms{E}$ whenever
$\ms{F} | \ms{E}$.  The collection of prime ends of $\Omega$ is called the {\em prime end 
boundary of $\Omega$} and is denoted $\partial_P \Omega$.  The set $\Omega$ together 
with its prime end boundary form the {\em prime end closure} of $\Omega$, denoted 
$\overline{\Omega}^P = \Omega \cup \partial_P \Omega$.
\end{defn}

We next describe a topology on $\overline{\Omega}^P$ using a notion of convergence.

\begin{defn}[Sequential topology for the prime end boundary]
We say that a sequence $\{x_k\}$ of points in $\Omega$ converges to a prime end $\ms{E}=[E_k]$ if for each positive integer $k$
there is a positive integer $N_k$ such that $x_j\in E_k$ for all $j\ge N_k$. It is possible for a sequence of points in $\Omega$
to converge to more than one prime end, see~\cite{ABBS}. 

A sequence of prime ends $\{\ms{E}_j\}$ is said to converge to a prime end $\ms{E}$ if, with each (or, equivalently, with some)
choice of $\{E_{k,j}\}\in\ms{E}_j$ and $\{E_k\}\in\ms{E}$, for each positive integer $k$ there is a positive integer $N_k$ such that
whenever $j\ge N_k$ there is a positive integer $m_{j,k}$ such that $E_{m_{j,k},j}\subset E_k$.
\end{defn}

It is shown in~\cite{ABBS} that the above notion of convergence yields a topology on $\overline{\Omega}^P$ which may
or may not be Hausdorff but satisfies the T1 separation axiom.

Recall that a domain $\Omega$ is of \emph{bounded turning} if there is some $\lambda\ge 1$ such that whenever
$x,y\in\Omega$ we can find a continuum $E_{x,y}\subset\Omega$ with $x,y\in E_{x,y}$ such that
$\diam(E_{x,y})\le \lambda d(x,y)$.

\begin{lemma}[Bounded Turning Boundary]
Let $X$ be a complete metric space.  Let $\Omega \subseteq X$ be a 
bounded domain with $\lambda$-bounded turning.  Then, $\overline{\Omega}^P$ is metrizable by an extension
of the Mazurkiewicz metric $\dum$ and there is a biLipschitz identification
$\overline{\Omega}=\overline{\Omega}^P$.
\end{lemma}

\begin{proof}
Let $x \in \partial \Omega$.  Let $x_i \in \Omega$ be such that $d(x_i, x) < 2^{-i}$.  
For each $i$, let $\gamma_i$ be a continuum with $x_i, x_{i+1} \in \gamma_i$ and 
$\diam(\gamma_i) \leq \lambda d(x_i, x_{i+1})$, and  let $\Gamma_j = \bigcup_{i=j}^\infty \gamma_i$.  
Then, $\Gamma_j$ is connected for each $j$ and $\diam(\Gamma_j) \leq \lambda 2^{-j+1}$.  Let $E_j$ 
be the connected component of $B(x, \lambda 2^{-j+2}) \cap \Omega$ containing $\Gamma_j$.  
Then $x\in\partial\Omega\cap\overline{E_j}$, and so $E_j$ is an acceptable set.

We claim $[E_j]$ is a prime end with $I(\{E_j\}) = \{x\}$.  We first show $\{E_j\}$ is a chain.  
Condition~(a) is clear as $\Gamma_{j+1} \subseteq \Gamma_j$ 
for each $j$, and $B(x, \lambda 2^{-j+1}) \subseteq B(x, \lambda 2^{-j + 2})$ for each $j$.  
If $y\in B(x,\lambda 2^{-j+2})\cap\Omega\cap\partial E_j$, then there is some $\rho>0$ such that
$B(y,\rho)\subset B(x,\lambda 2^{-j+2})\cap\Omega$, and there is a point 
$w\in B(y,\rho/[3\lambda])\cap E_j$. Thus whenever $z\in B(y,\rho/[3\lambda])$, by the bounded turning
properety of $\Omega$ we know that there is a continuum in $\Omega$ containing both $z$ and $w$, with
diameter at most $2\rho/3<\rho$, and so this continuum will lie inside $B(x,\rho)$. It follows that
$B(y,\rho/[3\lambda])\subset E_j$, violating the assumption that $y\in\partial E_j$. Therefore
$\Omega\cap \partial E_j\subset \partial B(x,\lambda 2^{-j+2})$. Hence
\[
\dM(\Omega\cap \partial E_j,\Omega\cap \partial E_{j+1})
\ge \dist(\Omega\cap \partial E_j,\Omega\cap \partial E_{j+1})\ge 2^{-j+1}>0,
\]
that is, Condition~(b) is valid.
Condition~(c) follows as $\{x\}=\bigcap_j\overline{E_j}$.  Hence, $\{E_j\}$ is a chain and
therefore $[E_j]$ is an end. That it is a prime end follows from~\cite{ABBS} together with the fact that
$I[E_j]$ is a singleton set.

%

We show that if $\mathscr{F}$ is a prime end with 
$I(\mathscr{F}) = \{x\}$, then $\mathscr{E} = \mathscr{F}$ where $\ms{E}$ is the one constructed above. 
To prove this it suffices
to show that $\ms{F} | \ms{E}$. Let $\ms{E}=[E_k]$ as above, and choose $\{F_k\}\in\ms{F}$. Fix a positive integer $k$
and let  $w_j\in \Omega\cap B(x,2^{-(2+k)})$ and $z_j\in E_k\cap B(x,2^{-(2+k)})$. Then by the bounded
turning property, there is a continuum $K_j$ containing both $w_j$ and $z_j$ such that 
$\diam(K_j)\le \lambda d(z_j,w_j)<\lambda 2^{-k-2}$. It follows that $K_j\subset B(x,\lambda 2^{-k+2})\cap\Omega$,
and so $K_j\subset E_k$. It follows that $B(x,2^{-k-2})\cap\Omega\subset E_k$. On the other hand,
$\bigcap_j\overline{F_j}=\{x\}$, and as $\overline{F_1}$ is compact with $F_{j+1}\subset F_j$ for each $j$, it follows that
for each $k$ there is some $j_k\in\N$ such that $F_{j_k}\subset B(x,2^{-k-2})\cap\Omega$. It then follows that
$F_{j_k}\subset E_k$. Therefore we have that $\{F_j\} | \{E_k\}$, that is, $\ms{F} | \ms{E}$ as desired.


We now show that if $\mathscr{F}$ is a prime end, then $I(\mathscr{F})$ consists of a single point. 
Let $x\in I(\ms{F})$ and choose $\{F_k\}\in\ms{F}$. Then for each positive integer $k$ we set
\[
\tau_k:=[3\lambda]^{-1}\min\{2^{-k-2}, \dM(\Omega\cap\partial F_j,\Omega\cap\partial F_{j+1})\, :\, j=1,\cdots, k\}.
\]
Note that $\tau_{k+1}\le \tau_k$ with $\lim_k\tau_k=0$.
We can then find a point $x_k\in B(x,\tau_k)\cap F_{k+1}$ since $x\in\partial F_j$ for each $j\in\N$. 
Note that as $F_{j+1}\subset F_j$ for each $j$, we necessarily have $x_{k+1}, x_k\in F_{k+1}$. By the bounded turning
property of $\Omega$ we can find a continuum $K_k\subset\Omega$ such that $x_k,x_{k+1}\in K_k$ and
$\diam(K_k)\le \lambda d(x_k,x_{k+1})<2\lambda\tau_k$. As
\[
2\lambda\tau_k<\dM(\Omega\cap\partial F_k,\Omega\cap\partial F_{k+1})
\]
and $x_k,x_{k+1}\in F_{k+1}$, it follows that $K_k\subset F_k$ for each positive integer $k$. A similar argument now also
shows that $B(x,\tau_k)\cap\Omega$ is contained in the connected component $G_k$ of $B(x,2\lambda\tau_k)\cap\Omega$
containing $K_k$, and hence $B(x,\tau_k)\cap\Omega\subset G_k\subset F_k$. It follows that $\{G_k\}$ is a chain in $\Omega$
with $\{G_k\} | \{F_k\}$, and as $\{F_k\}$ is a prime end, it follows that $[G_k]=\ms{F}$. Thus we have $I(\ms{F})=I([G_k])=\{x\}$.

The above argument shows that $\partial_P\Omega$ consists solely of singleton prime ends, and hence by
the results of~\cite{ABBS} we know that $\overline{\Omega}^P$ is compact and is
metrizable by an extension of the Mazurkiewicz
metric $\dum$. Moreover, for each $x\in\partial\Omega$ there is exactly one prime end $\ms{E}_x\in\partial_P\Omega$
such that $I(\ms{E}_x)=\{x\}$. To complete the proof of the lemma, note that if $x,y\in\Omega$, we can find a continuum
$E_{x,y}\subset\Omega$ with $x,y\in E_{x,y}$ such that $\diam(E_{x,y})\le \lambda d(x,y)$; therefore 
$d(x,y)\le \dum(x,y)\le \lambda d(x,y)$. Since $\overline{\Omega}^P$ is the completion of $\Omega$ with respect to the
metric $\dum$, this biLipschitz correspondence extends to $\overline{\Omega}\to\overline{\Omega}^P$ as wished for.
\end{proof}

Throughout this paper, $\Omega$ and $\Omega^\prime$ are domains in $X$. Recall that a set $E\subset X$
is a continuum if it is connected and compact. Such a set $E$ is nondegenerate if in addition it has at least
two points.

\begin{defn}\label{defn:finiteConnBdy}
A domain $\Omega$ is said to be \emph{finitely connected at} a point $x\in\partial\Omega$ if for every $r>0$ the following two
conditions are satisfied:
\begin{enumerate}
\item There are only finitely many connected components $U_1,\cdots U_k$ of $B(x,r)\cap\Omega$ with $x\in\partial U_i$
for $i=1,\cdots, k$,
\item there is some $\rho>0$ such that $B(x,\rho)\cap\Omega\subset \bigcup_{j=1}^k U_j$.
\end{enumerate}
\end{defn}

It was shown in~\cite{ABBS} (see also~\cite{BBS1}) that $\Omega$ is finitely connected at every point of its boundary if and only if 
every prime end of $\Omega$ has a singleton impression (that is, its impression has only one point) and
$\partial_P\Omega$ is compact.

It is an open problem whether, given a bounded domain, every end of that domain is divisible by a prime end,
see for example~\cite{ABBS, AS, ES}. Since this property is of importance in the theory of Dirichlet problem 
for prime end boundaries (see for instance~\cite{AS, ES}), the following useful lemma is also of independent interest.

\begin{lemma}[Divisibility]\label{lem:end-div-prime-finite-conn}
Suppose that $\Omega$ is a bounded domain that is finitely connected at $x\in\partial\Omega$, and let $[E_k]$ be an
end of $\Omega$ with $x\in I[E_k]$. Then there is a prime end $[G_k]$ such that $[G_k]\, |\, [E_k]$. 
\end{lemma}

\begin{proof}
For each positive integer $j$ let $C^j_1, \cdots, C^j_{k_j}$ be the connected components
of $B(x,2^{-j})\cap \Omega$ that contain $x$ in their boundary. For each choice of positive integers $j$ and $k$,
there is at least one choice of $m\in\{1,\cdots,k_j\}$ such that $C^j_m\cap E_k$ is non-empty. 

Let $V^\prime$ be the collection of all $C^j_m$, $j\in\N$ and $m\in\{1,\cdots, k_j\}$ for which $C^j_m$ intersects
$E_k$ for infinitely many positive integers $k$ (and hence, by the nested property of the chain $\{E_k\}$,
all positive integers $k$), and let $V:=V^\prime\cup\{\Omega\}$.  
As $\Omega$ is finitely connected at the boundary, it follows that for each $j$
there is at least one $m\in\{1,\cdots, k_j\}$ such that $C^j_m$ has this property. We set $C^0_1:=\Omega$.
For non-negative integer $j$ we say that $C^j_m$ is a 
neighbor of $C^{j+1}_n$ if and only if 
$C^{j+1}_n\subset C^j_m$. This neighbor relation creates a tree structure, with $V$ as its vertex set, with
each vertex of finite degree. 

Note that for each positive integer $M$, there is a path in this tree (starting from the root vertex
$C^0_1=\Omega$) with length at least $M$. Because each vertex has finite degree,
it follows that there is a path in this tree with infinite
length. To see this we argue as follows. Let property (P) be the property, applicable to vertices $C^j_m$, 
that the sub-tree with the vertex as its root vertex has arbitrarily long paths starting at that vertex.
For each non-negative integer $j$ the vertex $C^j_m$ has finite degree, and so if $C^j_m$ has property (P) then
it has at least one descendant neighbor (also known as a child in graph theory) $C^{j+1}_n$ with property (P).
Since $C^0_1=\Omega$ has property (P) as pointed out above, we know that there is a choice of
$m_1\in\{1,\cdots, k_1\}$ such that the vertex $C^1_{m_1}$ also has property (P). From here we can find
$m_2\in\{1,\cdots, k_2\}$ such that $C^1_{m_2}$ also has property (P) \emph{and} $C^1_{m_1}$ is a neighbor
of $C^2_{m_2}$. Inductively we can find $C^j_{m_j}$ for each positive integer $j$ to create this path.

Denoting this path by 
$\Omega\sim C^1_{m_1}\sim C^2_{m_2}\sim\cdots$, we see that the collection
$\{C^j_{m_j}\}$ is a chain for $\Omega$ with impression $\{x\}$. As a chain with singleton impression, it 
belongs to a prime end $[G_k]$. Recall that $X$ is locally path-connected and $\Omega\subset X$ is 
an open connected set. Therefore,
for each positive integer $i$ we see that there is some
$j_i$ such that $C^j_{m_j}\subset E_i$ when $j\ge j_i$, 
and so $[G_k]\, |\, [E_k]$ (for if not, then for each $j$ the open set
$C^j_{m_j}$ contains a point from $E_{i+1}$ and a point from $\Omega\setminus E_i$, and so 
$\dM(\Omega\cap\partial E_i,\Omega\cap\partial E_{i+1})\le 2^{1-j}$ for each $j$, violating the property of
$\{E_k\}$ being a chain).
\end{proof}

The following two lemmata regarding ends are useful to us.

\begin{lemma}\label{lem:open-end}
Let $\{E_k\}$ be a chain of $\Omega$. Then there is a chain $\{F_k\}$ of $\Omega$ that is equivalent to
$\{E_k\}$ such that for each $k\in\N$ the set $F_k$ is open.
\end{lemma}

\begin{proof}
For each positive integer $k$ we choose $F_k$ to be the connected component of $\text{int}(E_k)$ that 
contains $E_{k+1}$. It is clear that 
\begin{equation}\label{eq:AAA}
F_{k+1}\subset E_k\subset F_k
\end{equation} 
for each $k\in\N$ and that $F_k$ is
connected and open. Moreover, as $E_{k+1}\subset F_k$, it follows that 
$\bigcap_k\overline{E_k}\subset\bigcap_k\overline{F_k}\subset\bigcap_k\overline{E_k}\subset\partial\Omega$,
and so we have that
\begin{equation}\label{eq:AA}
\emptyset\ne\bigcap_k\overline{F_k}.
\end{equation}
We show now that $\partial F_k\subset\partial E_k$, from which it will follow that
\[
\dM(\Omega\cap\partial F_k,\Omega\cap\partial F_{k+1})>0
\]
because $E_k, E_{k+1}$ also satisfy analogous condition. To see~\eqref{eq:AA}, let $x\in\partial F_k$.
Then, for each $\eps>0$, the ball $B(x,\eps)$ intersects $F_k$, and so also intersects $E_k$. Hence
$x\in\overline{E_k}$. If $x\not\in\partial E_k$, then $x\in\text{int}(E_k)$, and in this case there is a ball
$B(x,\tau)\subset E_k$; we argue that this is not possible. Indeed, by the local path-connectedness of $X$
there must be a connected open set $U\subset\Omega$ such that $x\in U\subset B(x,\tau)\subset E_k$, and
as $x\in\partial F_k$, necessarily $U$ intersects $F_k$ as well. Then $F_k\cup U$ is connected, with
$F_k\cup U\subset E_k$; this is not possible as $F_k$ is the largest open connected subset of $E_k$.
Therefore $x\not\in\text{int}(E_k)$, and so $x\in\partial E_k$.

From the above we have that $\{F_k\}$ is a chain of $\Omega$. By~\eqref{eq:AAA} we also know that this chain
is equivalent to the original chain $\{E_k\}$.
\end{proof}

\begin{lemma}\label{lem:connected-impression}
If $\{E_k\}$ is a chain of $\Omega$, then its impression $I\{E_k\}$ is a compact connected subset of $X$.
\end{lemma}

\begin{proof}
For each $j\in\N$ set $K_j:=\overline{E_j}$. Then $K_j$ is compact, non-empty, and connected with 
$K_{j+1}\subset K_j$. Set $K:=\bigcap_jK_j$. Suppose that $K$ is not connected. Then there are open sets
$V,W\subset X$ with $K\subset V\cup W$, $K\cap V\ne\emptyset$, $K\cap W\ne \emptyset$, and 
$K\cap V\cap W$ empty. We set $K_V:=K\cap V$ and $K_W:=K\cap W$. Then $K_V=K\setminus W$ and 
$K_W=K\setminus V$, and so both $K_V$ and $K_W$ are compact sets. Moreover, 
$K_V\cap K_W=K\cap V\cap W=\emptyset$. Therefore 
\[
\dist(K_V,K_W)>0,
\]
and so there are open sets $O_V, O_W\subset X$ such that $K_V\subset O_V$, $K_W\subset O_W$, and
$O_V\cap O_W=\emptyset$. Set $U:=O_v\cup O_W$. As each $K_j$ is connected, it follows that we cannot
have $K_j\subset U$, for then we will have $K_j\subset O_V\cup O_W$ with $K_j\cap O_V\cap O_W$ empty
but $K_j\cap O_V\supset K\cap O_V\ne\emptyset$ and $K_j\cap O_W\supset K\cap O_W\ne \emptyset$.
It follows that for each $j\in\N$ there is a point $x_j\in K_j\setminus O$. This forms a sequence in
the compact set $K_1$, and so there is a subsequence $x_{j_n}$ that converges to some point $x\in K_1$.

For each $k\in\N$ we have that $K_k$ is compact and $\{x_k\}_{j\ge k}$ a sequence there; so
$x\in K_k$ for each $k\in\N$, that is, $x\in K$. However, $K\subset O$ and each $x_j\in X\setminus O$; therefore
we must have $x\in X\setminus O$ (recall that $O$ is open), leading us to a contradiction. Therfore
we conclude that it is not possible for $K$ to be not connected. 
\end{proof}


\begin{defn}[BQS homeomorphisms]
A homeomorphism $f:\Omega\to\Omega^\prime$ is said to be a \emph{branched quasisymmetric homeomorphism} (or BQS
homeomorphism) if there is a monotone increasing map $\eta:(0,\infty)\to(0,\infty)$ with
$\lim_{t\to0^+}\eta(t)=0$ such that whenever $E,F\subset\Omega$ are nondegenerate continua with $E\cap F$ non-empty, we have 
\[
\frac{\text{diam}(f(E))}{\text{diam}(f(F))}\le \eta\left(\frac{\text{diam}(E)}{\text{diam}(F)}\right).
\]
\end{defn}

\begin{defn}[Quasisymmetry]
A homeomorphism $f:\Omega\to\Omega^\prime$ is said to be
an $\eta$-quasisymmetry if $\eta:[0,\infty)\to[0,\infty)$ is a homeomorphism
and for each triple of distinct points $x,y,z\in\Omega$ we have
\[
\frac{d_Y(f(x),f(y))}{d_Y(f(x),f(z))}\le \eta\left(\frac{d_X(x,y)}{d_X(x,z)}\right).
\]
\end{defn}

Quasisymmetries are necessarily BQS homeomorphisms, but not all BQS homeomorphisms are quasisymmetric 
as the mapping $f$ in Example~\ref{ex:not-BQS-Plane-Sphere} shows.
However, if both $\Omega$ and $\Omega^\prime$ are of bounded turning, then every
BQS homeomorphism is necessarily quasisymmetric. BQS homeomorphisms
between two locally path-connected metric spaces continue to be
BQS homeomorphisms when we replace the original metrics with their respective Mazurkiewicz
metrics, see Proposition~\ref{prop:BQS-Mazur}. When equipped with the respective Mazurkiewicz
metrics, the class of BQS homeomorphisms coincides with the class of quasisymmetric maps. 
However, a domain equipped with its Mazurkiewicz metric could have a different prime end
structure than the prime end structure obtained with respect to the original metric,
see Example~\ref{ex:long-toothed-double-comb} for instance.
Hence, considering BQS maps as quasisymmetric maps between two metric spaces equipped with
a Mazurkiewicz metric would not give a complete picture of the geometry of the domain.

There are many notions of quasiconformality in the metric setting, the above two being examples of them. The following
is a geometric version.

\begin{defn}[Geometric quasiconformality]
Let $\mu_Y$ be an Ahlfors $Q$-regular measure on $Y$ and $\mu_X$ be a doubling measure on $X$.
A homeomorphism $f:\Omega\to\Omega^\prime$ is a \emph{geometrically quasiconformal} mapping if there is a constant 
$C\ge 1$ such that whenever $\Gamma$ is a family of curves in $\Omega$, we have
\[
C^{-1} \Mod_Q(f\Gamma)\le \Mod_Q(\Gamma)\le C \Mod_Q(f\Gamma).
\]
Here, by $\Mod_p(\Gamma)$ we mean the number
\[
\Mod_p(\Gamma):=\inf_\rho\int_\Omega\rho^p\, d\mu_X
\]
with infimum over all non-negative Borel-measurable functions $\rho$ on $\Omega$ for which $\int_\gamma\rho\, ds\ge 1$
whenever $\gamma\in\Gamma$.
\end{defn}

From the work of Ahlfors and Beurling~\cite{Ahl, Ahl1, AB}
we know that conformal mappings between two planar domains are necessarily
geometrically quasiconformal; a nice treatment of quasiconformal mappings on Euclidean domains
can be found in~\cite{Va}. 
Hence every simply connected planar domain that is not the entire complex plane is
geometrically quasiconformally equivalent to the unit disk. However, as Example~\ref{example:comb} shows, not every
conformal map is a BQS homeomorphism; hence among homeomorphisms between two Euclidean domains, 
the class of BQS homeomorphisms is smaller than the class of geometric quasiconformal maps. Furthermore, 
from Example~\ref{example:slit} we know that while quasisymmetric maps between two Euclidean domains are 
BQS homeomorphisms, not all BQS homeomorphisms are quasisymmetric. 
Thus the narrowest classification of
Euclidean domains is via quasisymmetric maps, the next narrowest is via BQS homeomorphisms. 
Among locally Loewner metric measure spaces,
geometric quasiconformal maps provide the widest classification. 
Therefore, even in classifying simply connected
planar domains, a deeper understanding of the geometry of the domains is gained by considering the
Carath\'eodory prime end compactification in tandem with the prime ends constructed in~\cite{ABBS, ES, E}.

\section{Bounded geometry and Loewner property}

Not all BQS homeomorphisms are quasisymmetric, as shown by Example~\ref{example:slit} from the previous 
section, and not all geometrically 
quasiconformal maps are BQS, as shown by Example~\ref{example:comb}. On the other hand,
BQS homeomorphisms need not be geometrically quasiconformal, as the next example shows.

\begin{ex}
Let $X=\R^n$ be equipped with the Euclidean metric and Lebesgue measure, and $Y=\R^n$ be equipped with the
metric $d_Y$ given by $d_Y(x,y)=\sqrt{\Vert x-y\Vert}$. Then $X$ has a positive modulus worth of families of
non-constant curves, while $Y$ has none. Therefore the natural identification map $f:X\to Y$ is not
geometrically quasiconformal; however, it is not difficult to see that $f$ is a BQS homeomorphism
and a quasisymmetric map.
\end{ex}

In this section we will describe notions of Ahlfors regularity and Poincar\'e inequality that 
are needed in the final section of this paper where geometric quasiconformal maps are 
studied.

\begin{remark}
Note that $\Omega$ is a domain in $X$, and hence it inherits the metric $d_X$ from $X$. However, $\Omega$ also has other
induced metrics as well, for example the Mazurkiewicz metric $\dum$ as  in Definition~\ref{def:Mazurk}. 
We denote the balls in $\Omega$ with 
respect to the Mazurkiewicz metric $\dum$, centered at $z\in\Omega$ and with radius $\rho>0$, by
$B_M(z,\rho)$. 
Since we assume
that $X$ is locally path connected, we know that the topology generated by $d_X$ in $\Omega$ and the topology generated by
$\dum$ are the same. To see this, observe that if $U\subset\Omega$ is open with respect to the metric 
$d_X$, then for each $x\in U$
there is some $r>0$ such that $B(x,2r)\subset U$. Then as $d_X\le \dum$, it follows that $B_M(x,r)\subset B(x,2r)\subset U$; that is, 
$U$ is open with respect to $\dum$. Next, suppose that $U\subset\Omega$ is open with respect to 
$\dum$, and let $x\in U$. Then there is some $\rho>0$ such that $B_M(x,3\rho)\subset U$. We can also
choose $\rho$ small enough so that $B(x,3\rho)\subset\Omega$. Then by the local path-connectedness
of $X$, there is an open set $W\subset B(x,\rho)$ with $x\in W$ such that $W$ is path-connected.
So for each $y\in W$ there is a path (and hence a continuum) $K\subset W$ with $x,y\in K$.
Note that then $\dum(x,y)\le \diam(K)\le 2\rho<3\rho$, that is, $y\in B_M(x,3\rho)\subset U$.
Therefore $W\subset U$ with $x\in W$ and $W$ open with respect to the metric $d_X$. It follows that
$U$ is open with respect to $d_X$ as well.
\end{remark}

From the above remark, given an interval $I\subset\R$, a map $\gamma:I\to\Omega$
is a path with respect to $d_X$ if and only if it is a path with respect to $\dum$.

\begin{defn}[Upper gradients]
Let $Z$ be a metric space and $d_\Omega$ a metric on $\Omega$.
Given a function $f:\Omega\to Z$, we say that a non-negative Borel function $g$ is an upper gradient of $f$ in $\Omega$
with respect to the metric $d_\Omega$ 
if whenever $\gamma$ is a non-constant compact rectifiable curve in $\Omega$, we have
\[
d_Z(f(x),f(y))\le \int_\gamma g\, ds,
\]
where $x,y$ denote the end points of $\gamma$ and the integral on the right-hand side is taken with respect to the arc-length 
on $\gamma$.
\end{defn}

\begin{lemma}
If $\gamma$ is a curve in $\Omega$, then $\gamma$ is rectifiable with respect to $\dum$ if and only if it is rectifiable with
respect to $d_X$; moreover, $\ell_{d_X}(\gamma)=\ell_{\dum}(\gamma)$. Furthermore, if $g$ is a non-negative Borel measurable
function on $\Omega$, then $g$ is an upper gradient of a function $u$ with respect to the metric $d_X$ if and only if it is
an upper gradient of $u$ with respect to $\dum$. Also, given a family $\Gamma$ of paths in $\Omega$ and $1\le p<\infty$,
the quantity $\Mod_p(\Gamma)$ is the same with respect to $d_X$ and with respect to $\dum$.
\end{lemma}

\begin{proof}
It suffices to prove the lemma for non-constant paths $\gamma:I\to \Omega$ with $I$ a compact interval in $\R$.
Since $d_X\le \dum$, it follows that $\ell_{d_X}(\gamma)\le \ell_{\dum}(\gamma)$ (whether $\gamma$ is rectifiable
or not, this holds). To show that converse inequality, let $t_0<t_1<\cdots<t_m$ be a partition of the
domain interval $I\subset\R$ of $\gamma$. Then for $j=0,\cdots, m-1$, note that
\[
\dum(\gamma(t_j),\gamma(t_{j+1}))\le \diam(\gamma\vert_{[t_j,t_{j+1}]})
\le \ell_{d_X}(\gamma\vert_{[t_j,t_{j+1}]}).
\]
It follows that
\[
\sum_{j=0}^{m-1}\dum(\gamma(t_j),\gamma(t_{j+1}))\le \ell_{d_X}(\gamma).
\]
Taking the supremum over all such partitions of $I$ gives $\ell_{\dum}(\gamma)\le \ell_{d_X}(\gamma)$.
Therefore, whether $\gamma$ is rectifiable or not (with respect to either of the two metrics), we have
$\ell_{d_X}(\gamma)=\ell_{\dum}(\gamma)$. The rest of the claims of the lemma now follow directly from
this identity.
\end{proof}

Given the above lemma, the notions of Moduli of path families and geometric quasiconformality of mappings can be taken
with respect to either of $d_X$ or $\dum$ (and respectively on the image side in the case of geometric quasiconformal
mappings). On the other hand, the notions of Poincar\'e inequality and Ahlfors regularity differ in general based on whether
we consider the metric $d_X$ or the metric $\dum$, for the balls are different with respect to these two metrics.

\begin{defn}[Poincar\'e inequality]
We say that $\Omega$ supports a $p$-Poincar\'e inequality with respect to the metric $d_\Omega$
if there are constants $C>0$ and $\lambda\ge 1$ such that whenever
$u:\Omega\to\R$ has $g$ as an upper gradient on $\Omega$ with respect to the metric $d_\Omega$
and whenever $x\in\Omega$ and $r>0$, we have
\[
\inf_{c\in\R}\frac{1}{\mu(B_\Omega)}\int_B|u-c|\, d\mu\le C\, r\, \left(\frac{1}{\mu(\lambda B_\Omega)}\int_{\lambda B}g^p\, d\mu\right)^{1/p},
\]
where $B_\Omega:=\{y\in\Omega\, :\, d_\Omega(x,y)<r\}$ and 
$\lambda B_\Omega:=\{y\in\Omega\, :\, d_\Omega(x,y)<\lambda r\}$.
\end{defn}

\begin{defn}[Ahlfors regularity]
We say that a measure $\mu$ is doubling on $\Omega$ with respect to a metric $d_\Omega$ if $\mu$ is Borel regular and
there is a constant $C\ge 1$ such that for each $x\in \Omega$ and $r>0$,
\[
\mu(B_\Omega(x,2r))\le C\, \mu(B_\Omega(x,r)).
\]
We say that $\mu$ is Ahlfors $Q$-regular (with respect to the metric $d_\Omega$) 
for some $Q>0$ if there is a constant $C\ge 1$ such that whenever $x\in\Omega$ and
$0<r<2\diam(\Omega)$ (with diameter taken with respect to the metric $d_\Omega$) if
\[
\frac{r^Q}{C}\le \mu(B_\Omega(x,r))\le C\, r^Q.
\]
\end{defn}

From the seminal work of Heinonen and Koskela~\cite[Theorem~5.7]{HeiK} we 
know that if a complete metric measure space $(X,d,\mu)$ is
Ahlfors $Q$-regular (with some $Q>1$), then it supports a $Q$-Poincar\'e inequality if and only
if it satisfies a strong version of the $Q$-Loewner property. 

\begin{defn}
We say that $(X,d,\mu)$ satisfies a \emph{$Q$-Loewner property} if there is a function
$\phi:(0,\infty)\to(0,\infty)$ such that whenever $E,F\subset X$ are two disjoint continua
(that is, connected compact sets with at least two points),
\[
\Mod_Q(\Gamma(E,F,X)\ge \phi(\Delta(E,F))
\]
where $\Gamma(E,F,X)$ is the collection of all curves in $X$ with one end point in $E$ and the
other in $F$, and 
\[
\Delta(E,F):=\frac{\dist(E,F)}{\min\{\diam(E),\diam(F)\}}
\]
is the relative distance between $E$ and $F$. The strong version of the $Q$-Loewner property 
is that 
\[
\liminf_{t\to 0^+}\phi(t)=\infty.
\]
\end{defn}

\begin{thm}[{\cite[Theorem~3.6]{HeiK}}]\label{thm:HeiK}
If $(X,d,\mu)$ is $Q$-Loewner for some $Q>1$ and $\mu$ is Ahlfors $Q$-regular, then we can 
take $\phi(t)=C\, \max\{\log(1/t), |\log(t)|^{1-Q}\}$ for some constant $C\ge 1$.
\end{thm}

\section{Extending BQS homeomorphisms between bounded domains to prime end closures}

The principal focus of this section and the next section is to obtain a Carath\'eodory-type
extension of BQS homeomorphisms between two domains. In this section we will focus on 
bounded domains, and in the next section we will study unbounded domains. Thus in this section
both $\Omega$ and $\Omega^\prime$ are bounded domains in proper, locally connected metric spaces.
Theorem~\ref{thm:BQS-fp-homeo} and Proposition~\ref{prop:bdd-domain-ends-match} are the two
main results of this section. However, Proposition~\ref{prop:BQS-Mazur} is of independent interest
as it demonstrates that BQS property and  geometric quasiconformality can be formulated 
equivalently with respect to $d_X, d_Y$ or with respect to $\dum,\dumm$; moreover, this proposition holds also for unbounded domains. 

To understand how BQS homeomorphisms transform prime ends, 
we first need the following two lemmata.

\begin{lemma}\label{lem:bdd-to-bdd}
Let $X$ and $Y$ be complete, doubling, locally path-connected metric spaces with $\Omega\subset X$,
$\Omega^\prime\subset Y$ two domains (not necessarily bounded). 
If there is a BQS homeomorphism $f:\Omega\to\Omega^\prime$, then
$\Omega$ is bounded if and only if $\Omega^\prime$ is bounded.
Moreover, if $A\subset\Omega$ with $\diam_M(A)<\infty$, then $\diam_M^\prime(f(A))<\infty$.
\end{lemma}

\begin{proof}
It suffices to show that if $\Omega$ is bounded, then $\Omega^\prime$ is also bounded.
We fix a continuum $\Gamma\subset\Omega$ such that $\diam(\Gamma)\ge \diam(\Omega)/2$, 
and let $\tau:=\diam(\Gamma)$, $s:=\diam(f(\Gamma))$.
Note that $0<s\tau<\infty$.
Let $z,w\in\Omega^\prime$. Then as $\Omega$ is locally path-connected, it follows that there
is a path (and hence a continuum) $\Lambda$ connecting $f^{-1}(z)$ to $f^{-1}(w)$.
The set $f(\Lambda)$ is a continuum in $\Omega^\prime$ with $z,w\in f(\Lambda)$. 
Let $\Gamma_1$ be the continuum obtained by connecting $\Gamma$ to $\Lambda$ in $\Omega$; then
$\diam(\Gamma_1)\ge \tau$ and so by the
BQS property of $f$ we now have
\[
\frac{\diam(f(\Lambda))}{\diam(f(\Gamma_1))}\le \eta\left(\frac{\diam(\Lambda)}{\tau}\right)
  \le \eta\left(\frac{\diam(\Omega)}{\tau}\right).
\]
Moreover, as $\Gamma$ and $\Gamma_1$ intersect with 
\[
\diam(\Omega)/2\le \diam(\Gamma)\le \diam(\Gamma_1)\le \diam(\Omega),
\]
we know from the BQS property of $f$ that
\[
\frac{\diam(f(\Gamma_1))}{s}\le \eta\left(\frac{\diam(\Omega)}{\diam(\Omega)/2}\right).
\]
It follows that 
\[
d_Y(z,w)\le \diam(f(\Lambda))
\le \diam(f(\Gamma_1))\, \eta\left(\frac{\diam(\Omega)}{\tau}\right)
 \le s\, \eta(2)\, \eta\left(\frac{\diam(\Omega)}{\tau}\right)<\infty,
\]
and so $\Omega^\prime$ is also bounded.

Now suppose that $A\subset\Omega$ with $\diam_M(A)<\infty$. If $A$ has only one point, then
there is nothing to prove. Hence assume that $0<\diam_M(A)$; then we can find a continuum
$\Gamma\subset\Omega$ intersecting $A$, such that 
\[
\diam_M(A)/2\le \diam(\Gamma)=\diam_M(\Gamma)\le 2\diam_M(A).
\]
Let $z,w\in f(A)$; then we can find a path $\beta\subset\Omega$ with end points $f^{-1}(z), f^{-1}(w)$
such that $\diam(\beta)\le 2\diam_M(A)$.
We can also find a continuum $\gamma\subset\Omega$ with $f^{-1}(w)$ and a point from
$A\cap\Gamma$ contained in $\gamma$ and $\diam(\gamma)\le 2\diam_M(A)$ 
(note that $f^{-1}(w)\in A\cap\beta$).
It follows from the BQS property of $f$ that
\[
\frac{\diam(f(\beta))}{\diam(f(\gamma\cup \Gamma))}
\le \eta\left(\frac{\diam(\beta)}{\diam(\gamma\cup \Gamma)}\right)
  \le \eta\left(\frac{2\diam_M(\Gamma)}{\diam(f(\Gamma))}\right).
\]
Applying the BQS property of $f$ to the two intersecting continua $\gamma\cup \Gamma$ and $\Gamma$ gives
\[
\frac{\diam(f(\gamma\cup \Gamma))}{\diam(f(\Gamma))}
\le \eta\left(\frac{\diam(\gamma\cup \Gamma)}{\diam(\Gamma)}\right)
 \le \eta(8).
\] 
It follows that
\[
\dumm(z,w)
\le \diam(f(\beta))\le \eta(8)\, \eta\left(\frac{2\diam_M(\Gamma)}{\diam(f(\Gamma))}\right).
\]
It follows that
\[
\diam_M^\prime(f(A))
\le \diam(f(\beta))\le \eta(8)\, \eta\left(\frac{2\diam_M(\Gamma)}{\diam(f(\Gamma))}\right)<\infty
\]
as desired.
\end{proof}

\begin{lemma}\label{lem:maz-dist-bdry-positive}
Let $X$ and $Y$ be complete, doubling, locally path-connected metric spaces with $\Omega\subset X$,
$\Omega^\prime\subset Y$ two domains (not necessarily bounded) and let $f:\Omega\to\Omega^\prime$ be
a BQS homeomorphism. Suppose that $E,F\subset\Omega$ are two
non-empty bounded connected open sets with $E\subset F$ such that
\begin{equation}\label{eq:pos-mzr}
\dM(\Omega\cap\partial E,\Omega\cap\partial F)>0.
\end{equation}
Then
\[
\dM^\prime(\Omega^\prime\cap\partial f(E),\Omega^\prime\cap\partial f(F))>0.
\]
\end{lemma}

\begin{proof}
Suppose $\gamma$ is a continuum in $\Omega^\prime$ connecting $\Omega^\prime\cap\partial f(E),\Omega^\prime\cap\partial f(F)$.
Then because $f$ is a homeomorphism, $f^{-1}(\gamma)$ is a continuum in $\Omega$ connecting
$\Omega\cap\partial E,\Omega\cap\partial F$, and hence
\[
\diam(f^{-1}(\gamma)\ge \dM(\Omega\cap\partial E,\Omega\cap\partial F):=\tau>0.
\]
Let $A\subset\Omega$ be a continuum in $F$ with $2\diam(F)\ge \diam(A)$; as $X$ is locally path connected
and $F$ is a non-empty bounded connected open set, such a continuum exists. 
As $f^{-1}\gamma$ intersects $F$ (for by the assumption~\eqref{eq:pos-mzr},
we know that $\Omega\cap\partial E\subset F$), we can find a continuum $\beta\subset F$ intersecting both 
$A$ and $f^{-1}(\gamma)\cap F$ such that $\diam(\beta)\le 2\diam(F)$. 
By applying the BQS property of $f$ to the two intersecting continua $A\cup\beta$ and $f^{-1}(\gamma)$, we have
\[
\frac{\diam(f(A\cup\beta))}{\diam(\gamma)}\le \eta\left(\frac{\diam(A\cup\beta)}{\diam(f^{-1}(\gamma))}\right)
  \le \eta\left(\frac{4\diam(F)}{\tau}\right).
\]
As $\diam(f(A\cup\beta))\ge \diam(f(A))>0$, we see that
\[
\diam(\gamma)\ge \frac{\diam(f(A))}{\eta\left(\frac{4\diam(F)}{\tau}\right)}>0,
\]
that is,
\[
\dM^\prime(\Omega^\prime\cap\partial f(E),\Omega^\prime\cap\partial f(F))
 \ge \frac{\diam(f(A))}{\eta\left(\frac{4\diam(F)}{\tau}\right)}>0
 \]
 as desired.
\end{proof} 

\begin{lemma}\label{lem:BQS-single-to-single}
Let $X$ and $Y$ be complete, doubling, locally path-connected metric spaces with $\Omega\subset X$,
$\Omega^\prime\subset Y$ two domains (not necessarily bounded) and let $f:\Omega\to\Omega^\prime$ be
a BQS homeomorphism. If $\{E_k\}$ is a chain of $\Omega$ with singleton impression, then
$\bigcap_k\overline{f(E_k)}$ has only one point.
\end{lemma}

\begin{proof}
Suppose that 
and let $\{E_k\}$ is a chain with singleton impression. Then $\lim_k\diam(E_k)=0$. We fix a continuum $J\subset E_1$ with
$\diam(E_1)/2\le \diam(J)\le \diam(E_1)$, and let $k\in\N$. For each pair $z,w\in f(E_k)$ we can
find a continuum $\gamma\subset E_k$ containing $f^{-1}(z)$ and $f^{-1}(w)$ such that
$\diam(\gamma)\le \diam(E_k)$. Let $J_k$ be a continuum in $E_1$ that intersects both $J$ and $\gamma$;
note that $\diam(J_k\cup J)\le \diam E_1$ as well.
Then by the BQS property of $f$, we obtain
\[
\frac{\diam(f(\gamma))}{\diam(f(J\cup J_k))}
  \le \eta\left(\frac{\diam(\gamma)}{\diam(J\cup J_k)}\right)\le \eta\left(\frac{\diam(E_k)}{\diam(J)}\right).
\]
On the other hand, by applying the BQS property of $f$ to the two intersecting continua $J$ and
$J\cup J_k$, we see that
\[
\frac{\diam(f(J\cup J_k))}{\diam(f(J))}
  \le \eta\left(\frac{\diam(J\cup J_k)}{\diam(f(J))}\right)\le \eta(2).
\]
Therefore
\[
\diam(f(\gamma))\le \eta(2)\, \eta\left(\frac{\diam(E_k)}{\diam(J)}\right).
\]
It follows that 
\[
\diam(f(E_k))\le \eta(2)\, \eta\left(\frac{\diam(E_k)}{\diam(J)}\right)\to 0\text{ as }k\to\infty.
\]
Therefore $\bigcap_k\overline{f(E_k)}$ has at most one point, and as $\overline{f(E_1)}$ is compact
and $\emptyset\ne\overline{f(E_{k+1})}\subset\overline{f(E_k)}$, it follows that 
$\bigcap_k\overline{f(E_k)}$ has exactly one point.
\end{proof}

\begin{thm}\label{thm:BQS-fp-homeo}
Let $X$ and $Y$ be complete, doubling metric spaces that are locally path connected.  
Let $\Omega \subseteq X$ and $\Omega' \subseteq Y$ be 
bounded domains.  Let $f \colon \Omega \to \Omega'$ be a branched quasi-symmetric (BQS) homeomorphism.  Then, 
$f$ induces a homeomorphism $f_P \colon \pec{\Omega} \to \pec{\Omega'}$.
\end{thm}

Note that both $\Omega$ and $\Omega^\prime$ are locally path connected because $X$ and $Y$ are; however,
we do not assume any further topological properties for these domains. For instance, we do not assume that either
$\Omega$ or $\Omega^\prime$ is finitely connected at their boundaries, and so we do not know that
$\overline{\Omega}^P$ is compact.

\begin{proof}
We must show that $f$ maps prime ends to prime ends.  

We first show that $f$ maps chains to chains.  As $f$ is a homeomorphism, $f$ maps admissible sets to 
admissible sets.  Indeed, if $E\subset\Omega$ is an admissible set, then $E$ is connected and so $f(E)$ is 
also connected. Suppose $\overline{f(E)}$ does not intersect $\partial\Omega^\prime$; then $\overline{f(E)}$ is
a compact subset of $\Omega^\prime$, and hence $f^{-1}(\overline{f(E)})$ is a compact subset of $\Omega$.
Therefore $\overline{E}\subset f^{-1}(\overline{f(E)})\subset\Omega$, which violates the requirement that
$\overline{E}\cap\partial\Omega$ be non-empty.

Suppose $\{E_i\}$ is a chain in $\Omega$.  
We check conditions $(a) - (c)$ of Definition~\ref{def:chains} for the sequence $\{fE_i\}$.  
Note that if $E \subseteq F \subseteq \Omega$, then $fE \subseteq fF \subseteq \Omega'$ as $f$ is a homeomorphism.  
Hence, $fE_{i+1} \subseteq fE_i$ for all $i$, so condition (a) holds.  
Condition~(b) follows from Lemma~\ref{lem:maz-dist-bdry-positive} by considering $E:=E_{k+1}$ and
$F:=E_k$.

Lastly, let $\{E_i\} \in \mathscr{E}$, and for each $i$ let $F_i = fE_i$.  
Suppose $y \in \Omega'$.  Let $x = f^{-1}(y) \in \Omega$.  As $\bigcap_{i=1}^\infty \overline{E_i} \subseteq \partial \Omega$, 
there is an $i_0$ and a $\rho > 0$ such that $B(x, \rho) \subseteq \Omega$ 
and $B(x, \rho) \cap \overline{E_{i_0}} = \emptyset$.  Then, 
$f(B(x,\rho))\cap f(E_{i_0})$ is empty. As $f(B(x,\rho))$ open and $y\in f(B(x,\rho))$, there is some
$\rho_1>0$ such that $B(y,\rho_1)\subset\Omega^\prime\cap f(B(x,\rho))$. It follows that
$\overline{f(E_{i_0})}\cap B(y,\rho_1/2)$ is empty, and hence $y\not\in\overline{f(E_{i_0})}$.
 Condition (c) follows.
 
 The above argument tells us that given a BQS homeomorphism $f:\Omega\to\Omega^\prime$, each chain $\{E_k\}$
 is mapped to a chain $f\{E_k\}:=\{f(E_k)\}$.  Next we show that ends are mapped to ends. To this end, note that if $\{G_k\}$ is also
 a chain in $\Omega$ with $\{G_k\} | \{F_k\}$, then for each $k$ there is a positive integer $i_k$ such that
 $G_{i_k}\subset F_k$. Hence $f(G_{i_k})\subset f(E_k)$, that is, $f\{G_k\} | f\{E_k\}$. Therefore division relation among
 chains is respected by $f$. Therefore, for each end $\ms{E}$ we have that $f\ms{E}\subset\ms{G}$ for some
 end $\ms{G}$ of $\Omega^\prime$. Given that $f^{-1}:\Omega^\prime\to\Omega$ is also a BQS homeomorphism, 
 we see that $f\ms{E}=\ms{G}$. Finally, if $\ms{E}$ is a prime end, then whenever $\{G_k\}$ is an end for $\Omega^\prime$
 with $\{G_k\} | f\ms{E}$, we have that $f^{-1}\{G_k\} | \ms{E}$ and hence as $\ms{E}$ is a prime end, we must have
 $f^{-1}\{G_k\}\in\ms{E}$. Therefore $\ms{E} | f^{-1}\{G_k\}$, from which it follows that $f\ms{E} | \{G_k\}$, that is,
 $\{G_k\}\in f\ms{E}$. Hence $f\ms{E}$ is also a prime end. Thus we conclude that whenever $\ms{E}\in\partial_P\Omega$,
 we must have $f\ms{E}\in\partial_P\Omega^\prime$, and we therefore have the natural extension
 $f_P:\overline{\Omega}^P\to\overline{\Omega^\prime}^P$ as desired. 
 
 The fact that $f_P$ is bijective follows from applying 
 the above discussion also to $f^{-1}$ (a BQS homeomorphism in its own right) to obtain an inverse of $f_P$.

To complete the proof of the theorem it suffices to show that $f_P$ is continuous, for 
then a similar argument applied to $f^{-1}$ and $f^{-1}_P$ yields continuity of $f_P^{-1}$.
To show continuity, it suffices to show that if 
$\{\zeta_k\}$ is a sequence in $\overline{\Omega}^P$ and $\zeta\in\overline{\Omega}^P$ such that 
$\zeta_k\to\zeta$, then $f\zeta_k\to f\zeta$. If $\zeta\in\Omega$, then the tail-end of the sequence also must lie in
$\Omega$, and thus the claim will follow from the fact that $f_P\vert_\Omega=f$ is a homeomorphism of $\Omega$.
Therefore it suffices to consider only the case $\zeta\in\partial_P\Omega$. By separating the sequence into two subsequences
if necessary, we can consider two cases, namely, that $\zeta_k\in\Omega$ for each $k$, or $\zeta_k\in\partial_P\Omega$ for
each $k$.

First, suppose that $\zeta_k\in\Omega$ for each $k$. Let $\{E_k\}\in\zeta$. Then the fact that $\zeta_k\to\zeta$ tells us that
for each $k\in\N$ we can find a positive integer $i_k$ such that whenever $i\ge i_k$, we have $\zeta_i\in E_k$.
It then follows that $f(\zeta_i)\in f(E_k)$. Thus for each $k$ there is a positive integer $i_k$ with $f(\zeta_i)\in f(E_k)$ when
$i\ge i_k$, that is, $f(\zeta_k)\to f([E_k])=f(\zeta)$.

Finally, suppose that $[F_{j,k}]=\zeta_k\in\partial_P\Omega$ for each $k$. We again choose any chain $\{E_k\}\in\zeta$. Then
for each $k$ there is a positive integer $i_k$ such that for each $i\ge i_k$, there exists a positive integer $j_{i,k}$ such that
whenever $j\ge j_{i,k}$ we have $F_{j,i}\subset E_k$. So we have $f(F_{j,i})\subset f(E_k)$ for
$j\ge j_{i,k}$, that is, $f([F_{j,k}])=f(\zeta_k)\to f([E_k])=f(\zeta)$. This completes the proof of the theorem.
\end{proof}

\begin{remark}\label{rmk:usefulness1}
The only place in the above proof where the BQS property plays a key role is in verifying that if
$E_k,E_{k+1}$ are two connected sets in $\Omega$ with $E_{k+1}\subset E_k$ and
\begin{equation}\label{eq:0}
\dM(\Omega\cap\partial E_k,\Omega\cap\partial E_{k+1})>0,
\end{equation} 
then we must have
\begin{equation}\label{eq:1}
\dM(\Omega^\prime\cap\partial f(E_k),\Omega^\prime\cap\partial f(E_{k+1}))>0. 
\end{equation}
This part was explicated into Lemma~\ref{lem:maz-dist-bdry-positive}.
The remaining parts of
the proof of existence and continuity of the map $f_P$ 
hold for \emph{any} homeomorphism $f:\Omega\to\Omega^\prime$ that satisfies~\eqref{eq:1}
for each pair of connected sets 
$E_{k+1}\subset E_{k}\subset \Omega$ that satisfies~\eqref{eq:0}. This is useful to 
keep in mind in the final section, where we consider quasiconformal mappings that are not necessarily
BQS maps.
\end{remark}


In the situation of the above theorem, if $\overline{\Omega}^P$ is compact, then so is
$\overline{\Omega^\prime}^P$. This is a useful property in determining that two domains are not BQS-homeomorphically
equivalent, as the following example shows.

\begin{ex}\label{example:comb}
With $\Omega$ a planar disk and $\Omega^\prime$ given by
\[
\Omega^\prime:=(0,1)\times(0,1)\setminus \bigcup_{n\in\N}\{1/n\}\times[0,1/2],
\]
we know from the Riemann mapping theorem that $\Omega$ and $\Omega^\prime$ are conformally equivalent. 
However, as $\overline{\Omega}^P$ is the closed disk, which is compact, and as
$\overline{\Omega^\prime}^P$ is not even sequentially compact, it follows that there is \emph{no}
BQS homeomorphism between these two planar domains.
\end{ex}

The above theorem is useful in other contexts as well.

\begin{prop}\label{prop:bdd-domain-ends-match}
Suppose that $X$ and $Y$ are locally path connected and $\Omega\subset X$, $\Omega^\prime\subset Y$ are
bounded open connected sets. Suppose that there is a BQS homeomorphism $f:\Omega\to\Omega^\prime$. 
\begin{enumerate}
\item[(a)] If $\overline{\Omega}^P$ is compact, then so is $\overline{\Omega^\prime}^P$.
\item[(b)] If $\overline{\Omega}^P$ is metrizable, then so is $\overline{\Omega^\prime}^P$.
\item[(c)] If $\Omega$ is finitely connected at the boundary, then so is $\Omega^\prime$.
\end{enumerate}
\end{prop}

\begin{proof}
By Theorem~\ref{thm:BQS-fp-homeo}, homeomorphically invariant properties of $\overline{\Omega}^P$
will be inherited by $\overline{\Omega^\prime}^P$; hence the first two claims of the proposition are
seen to be true. On the other hand, finite connectivity of a domain at its boundary is not a 
homeomorphic invariant. 
Hence we provide a proof of the last claim of the proposition here.

Note by~\cite{BBS1} that a bounded domain is finitely connected at the boundary if and only if all of
its prime ends are singleton prime ends and its prime end closure is compact. Thus if $\Omega$ is
finitely connected at the boundary, then $\overline{\Omega}^P$ is compact, and hence so is
$\overline{\Omega^\prime}^P$. Hence to show that $\Omega^\prime$ is finitely connected at the 
boundary, it suffices to show that all of the prime ends of $\Omega^\prime$ are singleton prime ends.
This is done by invoking Lemma~\ref{lem:BQS-single-to-single}.
\end{proof}

\begin{ex}\label{ex:long-toothed-double-comb}
It is possible to have two homeomorphic equivalent domains with one finitely connected at the boundary and
the other, not finitely connected at the boundary. Let 
\[
\Omega:=(0,1)\times(0,1)\setminus\bigcup_{n\in\N}\left([0,1-1/n]\times\{1/2n\}\right)\cup\left([1/n,1]\times\{1/(2n+1)\}\right)
\]
and $\Omega^\prime$ be the planar unit disk. Then as $\Omega$ is simply connected and bounded, it follows from
the Riemann mapping theorem that $\Omega$ and $\Omega^\prime$ are conformally, and hence homeomorphically,
equivalent. However, $\Omega^\prime$ is finitely connected at the boundary and $\Omega$ is not.
\end{ex}

\begin{ex}
The domain $\Omega$ constructed in Example~\ref{ex:long-toothed-double-comb} has the property that
$\overline{\Omega}^P$ is compact and metrizable. 
This is seen by the fact that $\overline{\Omega}^P$ is homeomorphic
to $\overline{U}^P$ where $U:=\varphi(\Omega)$ where $\varphi:\R^2\to\R^2$ is given by
$\varphi(x,y)=(xy,y)$. Observe that $U$ is simply connected at its boundary, and so 
$\overline{U}^P$ is metrizable.

On the other hand, the domain $\Omega^\prime$ given by
\[
\Omega^\prime:=(0,1)\times(0,1)\setminus\bigcup_{n\in\N}\left([0,2/3]\times\{1/2n\}\right)\cup
\left([1/3,1]\times\{1/(2n+1)\}\right)
\]
has the property that $\overline{\Omega^\prime}^P$ is not compact. 
Hence $\Omega$ is not BQS homeomorphically equivalent to $\Omega^\prime$, even
though both are simply connected planar domains that are not finitely connected at the boundary.
\end{ex}


\begin{ex}\label{example:slit}
With $\Omega$ the planar domain given by 
\[
\Omega:=\{z\in\C\, :\, |z|<1\text{ and }\text{Im}(z)>0\}
\] 
and
$\Omega^\prime:=\{z\in\C\, :\, |z|<1\}\setminus [0,1]$ a slit disk, the map
$f:\Omega\to\Omega^\prime$ given by $f(z)=z^2$ is a conformal equivalence that is not a quasisymmetric
map. However, this map is a BQS homeomorphism, and so has an extension
$f_P:\overline{\Omega}=\overline{\Omega}^P\to\overline{\Omega^\prime}^P$.
This follows from~\cite[Theorem~6.50]{GW2}.
\end{ex}

\begin{ex}
With $\Omega$ and $\Omega^\prime$ as in Example~\ref{example:slit} above, the map
$h:\Omega\to\Omega^\prime$ given by $h(re^{i\theta})=re^{2i\theta}$ is also a BQS homeomorphism which is
not conformal. This again follows from~\cite[Theorem~6.50]{GW2}.
\end{ex}

As indicated by the definition of chains in the prime end theory, a domain $\Omega$ can also naturally
be equipped with the Mazurkiewicz metric $\dum$. 

\begin{prop}\label{prop:BQS-Mazur}
Let $X$ and $Y$ be locally path-connected metric spaces, and 
$f:\Omega\to\Omega^\prime$ be a homeomorphism. Then $f$ is BQS with respect to the respective
metrics $d_X$ and $d_Y$ on 
$\Omega$ and $\Omega^\prime$ if and only if it is BQS with respect to the corresponding Mazurkiewicz metrics
$\dum$, $\dumm$.
\end{prop}

\begin{proof}
To prove the theorem it suffices to show that whenever $E\subset\Omega$ is a continuum, we have
$\diam(E)=\diam_M(E)$. Note that $x,y\in\Omega$ we have
$d(x,y)\le \dum(x,y)$. To prove the reverse of this inequality, let
$E\subset\Omega$ be a continuum, $\varepsilon>0$, and we choose $x,y\in E$ such that 
$[1+\varepsilon]\dum(x,y)\ge \diam_M(E)$. By the definition of the Mazurkiewicz metric $\dum$ (see Definition~\ref{def:Mazurk}),
we have
\[
\dum(x,y)\le \diam(E).
\]
It follows that
\[
\diam_M(E)\le [1+\varepsilon]\diam(E),
\]
The desired claim now follows by letting $\varepsilon\to 0^+$.
\end{proof}

\begin{remark}\label{rem:fp-is-BQS}
By the above proposition together with~\cite[Proposition~10.11]{Hei}, if $\Omega$ and
$\Omega^\prime$ are finitely connected at their respective boundaries
and $f:\Omega\to\Omega^\prime$ is a BQS homeomorphism, then its extension 
$f_P:\overline{\Omega}^P\to\overline{\Omega^\prime}^P$ is a BQS homeomorphism with respect to 
the respective Mazurkiewicz metrics $\dum$ and $\dumm$.
\end{remark}

Given Proposition~\ref{prop:BQS-Mazur} and the above remark, it is natural to ask why we do not consider
prime ends as constructed with respect to the Mazurkiewicz metric induced by the original metric $d_X$
on $\Omega$. However, while the property that a set $E\subset\Omega$ be connected
is satisfied under both $d_X$ and the induced Mazurkiewciz metric $\dum$, the crucial property
that $\overline{E}\cap\partial\Omega$ be non-empty depends on the metric (since the closure
$\overline{E}$ depends on the metric), and hence we have fewer acceptable sets and chains, and therefore
prime ends, under the induced metric $\dum$ than under the original metric $d_X$. 
See Lemma~\ref{lem:prime-end-dM} below for the limitations put on prime ends under $\dum$.
For instance, the 
domain $\Omega$ as in Example~\ref{ex:long-toothed-double-comb} will have as prime ends under $\dum$ 
only the singleton prime ends, with no prime end with impression containing any point $(x,0)$, $0\le x\le 1$.
Hence under the metric $\dum$ the prime end closure of $\Omega$ will not be compact. The issue of 
considering prime ends with respect to the original metric $d_X$ versus the induced metric $\dum$
becomes even more of a delicate issue when $\Omega$ is not bounded, see Section~\ref{sec:unbounded} below.

\begin{lemma}\label{lem:prime-end-dM}
Let $\Omega$ be a bounded domain in a locally path connected complete
metric space $X$, and let $\dum$ be
the Mazurkiewicz metric on $\Omega$ induced by the metric $d_X$ inherited from $X$. If
$[\{E_k\}]$ is a prime end in $\Omega$ with respect to $\dum$, then $[\{E_k\}]$ is a 
prime end with
respect to $d_X$ with $\bigcap_k\overline{E_k}$ a singleton set.
\end{lemma}

\begin{proof}
Note first that if $E\subset\Omega$ with $\overline{E}^M$ compact, then $\overline{E}^M$ is 
sequentially compact and hence so is $\overline{E}$; hence it follows that if $\overline{E}^M$
is compact, then $\overline{E}$ is also compact. Combining this together with the fact that
$\diam=\diam_M$ for connected sets, tells us that if $\{E_k\}$ is a chain with respect to $\dum$ then
it is a chain with respect to $d$.

Next, note that if $\{E_k\}$ is a chain with respect to $\dum$ and $\{G_k\}$ is a chain with
respect to $d$ such that $\{G_k\}\, |\, \{E_k\}$, then, by passing to a sub-chain if necessary, 
$\{G_k\}$ is also a chain with respect
to $\dum$. Indeed, the only property that needs to be verified here is whether each
$\overline{G_k}^M$ is compact. By the divisibility assumption, we know that there is some 
$j_0\in\N$ such that whenever $j\ge j_0$ we have $G_j\subset E_1$. Then $\overline{G_j}^M$ is
a closed subset of the compact set $\overline{E_1}^M$, and hence is compact.

Since $E_k$ is acceptable with respect to $\dum$, we know that $\overline{E_k}^M$, the closure of
$E_k$ under the metric $\dum$, is compact. It then follows also that 
$\bigcap_k\overline{E_k}^M$ is non-empty and has a point $\zeta\in\partial_M\Omega$. We can then
find a sequence $x_k\in E_k$ with $\lim_k \dum(x_k,\zeta)=0$. Since $d_X\le \dum$ on $\Omega$,
it follows that $\{x_k\}$ is a Cauchy sequence with respect to the metric $d_X$, and as
$X$ is complete, there is a point $x\in \overline{\Omega}$ such that $\lim_kd_X(x_k,x)=0$.
Since $\{x_k\}$ is also a Cauchy sequence with respect to the metric $\dum$, for each
positive integer $k$ we can find a continuum $\Gamma_k$ in $\Omega$ with end points
$x_k$ and $x_{k+1}$ such that (by passing to a subsequence if necessary) 
\[
\diam_M(\Gamma_k)=\text{diam}(\Gamma_k)\le 2^{-(k+1)}.
\]
In the above, the first identity is inferred from the proof of 
Proposition~\ref{prop:BQS-Mazur} above. For each positive ineger $k$ let
$G_k$ be the connected component of $B(x,2^{1-k})\cap\Omega$ containing the connected set
$\bigcup_{j\ge k}\Gamma_j\subset\Omega$. 
Then $\{G_k\}$ is a chain in $\Omega$ with respect to the metric $d$, and
$\{x\}=I(\{G_k\})$. It is clear from the second paragraph above
that $\{G_k\}$ is a chain also with respect to
$\dum$, with $\{\zeta\}=\bigcap_k\overline{G_k}^M$. As a chain with singleton impression, it
follows that the end (with respect to $\dum$) that $\{G_k\}$ belongs to is a prime end.
Moreover, for each positive integer $k$ we know that infinitely many of the sets $G_j$, $j\in\N$,
intersect $E_k$; it follows from the assumption that $[\{E_k\}]$ is a prime end that
$[\{E_k\}]=[\{G_k\}]$ (see the last part of the proof of
Lemma~\ref{lem:end-div-prime-finite-conn}), and so the claim follows. 
\end{proof}

\section{Prime ends for unbounded domains and extensions of BQS homeomorphisms}\label{sec:unbounded}

In this section we consider unbounded domains, and so we adopt the modification to the construction
of ends found in~\cite[Definition~4.12]{E}. However, we introduce a slight modification, namely
the boundedness of the separating compacta $R_k$ and $K$ wiht respect to the Mazurkiewicz metric $\dum$.
Recall that a metric space is proper if closed and bounded subsets of the
space are compact.

\begin{defn}\label{def:type-ab-ends}
A connected set $E\subset\Omega$ is said to be acceptable if $\overline{E}$ is proper and 
$\overline{E}\cap\partial\Omega$ is non-empty. A sequence $\{E_k\}$ of acceptable subsets of 
$\Omega$ is a chain if for each positive integer $k$ we have
\begin{enumerate}
  \item[(a)] $E_{k+1}\subset E_k$,
  \item[(b)] $\overline{E_k}$ is compact, or $E_k$ is unbounded and there is a compact set $R_k\subset X$ such that
 $\Omega\setminus R_k$ has two components $C_{1,k}$ and $C_{2,k}$ with 
 $\Omega\cap\partial E_k\subset C_{1,k}$ and $\Omega\cap\partial E_{k+1}\subset C_{2,k}$, and
 $\diam_M(R_k\cap\Omega)$ is finite,
  \item[(c)] $\dM(\Omega\cap\partial E_k,\Omega\cap\partial E_{k+1})>0$,
  \item[(d)] $\emptyset\ne\bigcap_k\overline{E_k}\subset\partial\Omega$.
 \end{enumerate}
\end{defn}

The above notion of ends agrees with the notion of ends from Section~\ref{sec:bdd-prime}
when $\Omega$ is bounded, for then the properness of $\overline{E}$ will guarantee that
$\overline{E}$ is compact. When $\Omega$ is not bounded, the above definition gives us two
types of ends,~(a) those that are bounded (i.e.~$E_k$ is bounded for some positive integer $k$)
and (b)~those that are unbounded (i.e.~$E_k$ is unbounded for each $k$). 
While the definition above agrees in philosophy with that from~\cite{E}, but the
requirement that $\diam_M(R_k\cap\Omega)$ be finite eliminates some of the candidates for
ends from~\cite{E} here. However, ends in our sense are necessarily ends in the sense of~\cite{E}.
From~\cite[Section~4.2]{E} we know that if $X$ itself is proper, then
ends (and prime ends, see below) as given in~\cite{E} are invariant under sphericalization 
and flattening procedures, and so understanding ends of type~(b) above can be accomplished by 
sphericalizing the domain and understanding ends of the resulting bounded domain with impression
containing the new point that corresponds to $\infty$.
Both of these types of
ends have non-empty impressions. 
The next definition deals with ends that should philosophically
be ends with only $\infty$ in their impressions.

\begin{defn}\label{def:type-c-ends}
An unbounded connected set $E\subset\Omega$ is said to be \emph{acceptable at $\infty$} if $\overline{E}$
is proper and there is a compact set $K\subset X$ 
such that
$\diam_M(K\cap\Omega)<\infty$ and 
$E$ is a component of $\Omega\setminus K$. A sequence $\{E_k\}$ of sets that are acceptable
at $\infty$ is said to be a \emph{chain at $\infty$} if for each positive integer $k$ we have
\begin{enumerate}
  \item[(a)] $E_{k+1}\subset E_k$,
  \item[(b)] $\dM(\Omega\cap\partial E_k,\Omega\cap\partial E_{k+1})>0$,
  \item[(c)] $\bigcap_k\overline{E_k}=\emptyset$.
 \end{enumerate}
\end{defn}

As with chains and ends for bounded domains in Section~\ref{sec:bdd-prime}, we define
divisibility of one chain by another, and say that two chains are equivalent if they divide
each other. An end is an equivalence class of chains, and a prime end is an end that is divisible
only by itself.

It was noted in~\cite[Remark~4.1.14]{E} that every end at $\infty$ is necessarily a prime end.

We now return to some mechanisms that make unbounded ends from Definition~\ref{def:type-ab-ends} 
work.

\begin{lemma}\label{lem:E1-E2-separation}
Let $E_1$ and $E_2$ be two open connected subsets of $\Omega$ such that $E_2\subset E_1$. 
Suppose in addition that $R$ is a
compact subset of $X$ such that no component of $\Omega\setminus R$ contain points from both
$\Omega\cap\partial E_1$ and $\Omega\cap\partial E_2$. 
Let $R_0:=R\cap\overline{E_1}\setminus E_2$. Then
$E_2$ is contained in a component $\widehat{C}$ of $\Omega\setminus R_0$ and $\widehat{C}\subset E_1$.
Moreover, no component of $\Omega\setminus R_0$ contains points from both
$\Omega\cap\partial E_1$ and $\Omega\cap\partial E_2$. 
\end{lemma}

\begin{proof}
Let $C_1$ be the union of all the components of $\Omega\setminus R_0$ that contain 
points from $\Omega\cap\partial E_1$, and let $C_2$ be the union of all the components of
$\Omega\setminus R_0$ that contain points from $\Omega\cap\partial E_2$.

Since $\Omega\cap\partial E_2\subset C_2$ and $C_2$ is open, it follows that $C_2\cup E_2$ is
a connected subset of $\Omega\setminus R_0$. So there is a component $\widehat{C}$ of 
$\Omega\setminus R_0$ such that $C_2\cup E_2\subset\widehat{C}\subset \Omega\setminus R_0$,
and hence $C_2=\widehat{C}\supset E_2$.

Let $U$ be a component of $\Omega\setminus R_0$ 
such that $U\subset C_1$. If $U\cap\widehat{C}$ is non-empty, then $U=\widehat{C}$, and in
this case there are two points $z\in U\cap\partial E_1$ and $w\in\widehat{C}\cap\partial E_2$
such that $z,w\in U$. As $U$ is an open and connected subset of $\Omega$, we can find a curve
$\gamma:[0,1]\to U$ with $\gamma(0)=z$ and $\gamma(1)=w$. We can assume, by considering a 
subcurve of $\gamma$ and modifying $z$ and $w$ if necessary, that 
$\gamma((0,1))\cap\partial E_1$ and $\gamma((0,1))\cap\partial E_2$ are empty.
As $\Omega\cap\partial E_2\subset E_1$, it follows that 
$\gamma((0,1))\subset \overline{E_1}\setminus E_2$.
By the property of separation that $R$ has, we must have $\gamma((0,1))\cap R$ non-empty.
Hence $\gamma((0,1))$ intersects $R\cap \overline{E_1}\setminus E_2=R_0$. This contradicts the
fact that $\gamma\subset U$ with $U\cap R_0$ empty. We therefore conclude that $C_1\cap\widehat{C}$
is empty.


Finally we are ready to prove that $\widehat{C}\subset E_1$. Suppose there is a point 
$z\in\widehat{C}\setminus E_1$. Since $E_2\subset\widehat{C}$ and
$\widehat{C}$ is open and connected, there is a point $w\in E_2\subset E_1$ and a curve $\beta\subset\widehat{C}$
connecting $z$ to $w$. As $z\not\in E_1$, it follows that $\gamma$ intersects $\Omega\cap\partial E_1$,
which violates the conclusion reached in the previous paragraph that $\Omega\cap\partial E_1$ does
not intersect $\widehat{C}$. Hence such $z$ cannot exist, whence it follows that $\widehat{C}\subset E_1$.
\end{proof}

As mentioned above, by~\cite[Section~4.2]{E}, understanding prime ends for unbounded domains
can be done by first sphericalizing them and transforming them into bounded domains and then
studying the prime ends of the transformed domains. There is a metric analog of reversing
sphericalizations, called flattening.  
However, the category of BQS homeomorphisms need not be preserved by the
sphericalization and flattening transformations, see the example below.  Therefore
sphericalization and flattening may not provide as useful a tool as we like in the study
of BQS maps.

\begin{ex}[BQS Sphere $\nRightarrow$ BQS Plane]
Let $f:\C^*=\C\setminus\{0\}\to\C^*$ be given by $f(z)=1/z$. The push-forward of $f$ under
sphericalization of $C^*$ to the twice-punctured sphere gives a BQS homeomorphism, but
$f$ is not a BQS homeomorshim from $C^*$ to itself, for the families of pairs of continua
$E_r:=[r,1]$ and $F_r:=\mathbb{S}^1$, $0<r<1/10$, have the property that
\[
\frac{\diam(E_r)}{\diam(F_r)}\approx 1, \qquad \frac{\diam(f(E_r))}{\diam(f(F_r))}\approx \frac{1}{r}.
\]
\end{ex}

\begin{ex}[BQS Plane $\nRightarrow$ BQS Sphere]\label{ex:not-BQS-Plane-Sphere}
For each positive integer $k$ let 
\begin{equation*}
\begin{split}
A_k = (\{\tfrac{1}{2^k}\} \times [1, 2^k]) \cup ([\tfrac{1}{2^{k+1}}, \tfrac{1}{2^k}]) \times \{2^k\}) \cup (\{\tfrac{1}{2^{k+1}}\} \times [0, 2^k]) \\
\cup ([\tfrac{1}{2^{k+2}}, \tfrac{1}{2^{k+1}}] \times \{0\}) \cup (\{\tfrac{1}{2^{k+2}}\} \times [0, 1]),
\end{split}
\end{equation*}
let $\eps_k = \tfrac{1}{64^k}$, $\ell(A_k)=2^{k+1}+\tfrac{3}{2^{k+2}}$ the length of $A_k$, and let
\[
\Omega:=\bigcup_{j\in\N}\bigcup_{z\in A_{2j}}Q(z,\eps_{2j}),
\]
where, for $z\in\C$, the cube 
$Q(z,t)=(\text{Re}(z)-t,\text{Re}(z)+t)\times(\text{Im}(z)-t,\text{Im}(z)+t)$. 
\begin{figure}[!htb]
\begin{center}
\begin{tikzpicture}
\draw[black, thick] (-1,1) -- (-1, 4);
\draw[black, thick] (-1,4) -- (-1.5, 4);
\draw[black, thick] (-1.5,4) -- (-1.5, 0);
\draw[black, thick] (-1.5,0) -- (-1.75, 0);
\draw[black, thick] (-1.75,0) -- (-1.75, 1);

\draw[black, dashed] (1,1) -- (1, 4);
\draw[black, dashed] (1,4) -- (.5, 4);
\draw[black, dashed] (.5,4) -- (.5, 0);
\draw[black, dashed] (.5,0) -- (.25, 0);
\draw[black, dashed] (.25,0) -- (.25, 1);

\draw[black] (1.0625,0.9375) -- (1.0625, 4.0625);
\draw[black] (1.0625,4.0625) -- (.4375, 4.0625);
\draw[black] (.4375,4.0625) -- (.4375, 0.0625);
\draw[black] (.4375,0.0625) -- (.3125, 0.0625);
\draw[black] (.3125,0.0625) -- (.3125, 1.0625);

\draw[black] (.9375,0.9375) -- (.9375, 3.9375);
\draw[black] (.9375,3.9375) -- (.5625, 3.9375);
\draw[black] (.5625,3.9375) -- (.5625, -.0625);
\draw[black] (.5625,-.0625) -- (.1875, -.0625);
\draw[black] (.1875,-.0625) -- (.1875, 1.0625);
\end{tikzpicture}
\end{center}
\caption{$A_k$ and its cube neighborhood}
\end{figure}
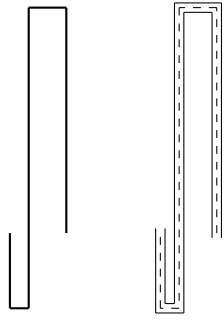

For $m \geq 1$, let $a_m = \sum_{j=1}^{m} \ell(A_{2j})$,  
and set $a_0 = 0$. 
Let 
\[
I_m = [a_{m-1}, a_m]=[a_{m-1}, a_{m-1}+\ell(A_{2m)}], \ I = [0, \infty).
\]  
Let $\Omega^\prime = \bigcup_{m=1}^\infty I_m \times (-\eps_{2m}, \eps_{2m})$.  
Then, there is a BQS map $f \colon \Omega^\prime \to \Omega$ defined  
as follows.  On $I$, this map is the arclength parametrization 
of the curve $\bigcup_k A_{2k}$.  On most of $\Omega^\prime$, this map is defined similarly; 
for $m\in\N$ and $\eps_{2m+2}\le\tau<\eps_{2m}$, the line 
$(0,a_m)\times\{\tau\}$ is mapped by $f$ as a parametrization to a curve that is an 
$\ell^\infty$-distance $\tau$ ``above" the curve $\bigcup_kA_{2k}$ (the dashed straight line gets
mapped to the dashed bent line above $f(I)$ in Figure~2 below).
This parametrization is mostly an arclength parametrization, 
although near the corners of $\Omega$ we change speed of this parametrization by an 
amount that is determined by $\tau$. We can ensure that this speed is at least $1/128$
and $128$.
For example, in the picture below, the dashed 
curve on the left is mapped to the dashed curve on the right by $f$ at a 
speed that is the ratio of the lengths of the dashed curve on the right and the 
middle curve on the right.

\begin{figure}[!htb]
\begin{center}
\begin{tikzpicture}
\draw[black, thick] (-10, 1.5) -- (-6, 1.5);
\draw[black, thick] (-10, -.5) -- (-6, -.5);
\draw[black, dashed] (-10, 1) -- (-6, 1);
\draw[black] (-10, .5) -- (-6, .5);

\node at (-3.5,.875) {$f$};
\node at (-5.8,.5) {$I$};
\node at (2.4,0) {$f(I)$};

\draw[->] (-5,.5) -- (-2,.5);

\draw[black] (0,0) -- (2,0);
\draw[black] (0,2) -- (0,0);
\draw[black, thick] (1,1) -- (2,1);
\draw[black, thick] (1,2) -- (1,1);
\draw[black, thick] (-1,-1) -- (-1, 2);
\draw[black, thick] (-1,-1) -- (2, -1);
\draw[black, dashed] (.5, .5) -- (2, .5);
\draw[black, dashed] (.5, .5) -- (.5, 2);
\end{tikzpicture}
\end{center}
\caption{Fiber map of $f$ near corners}
\end{figure}
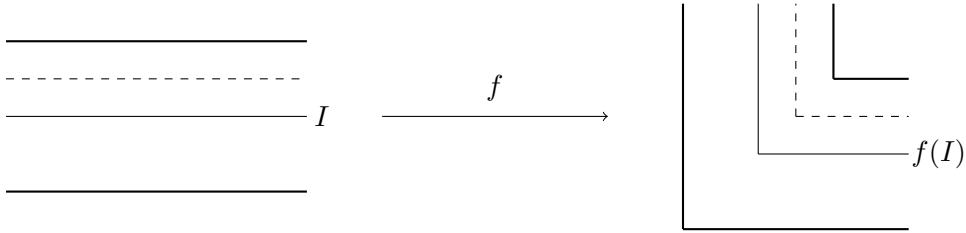

By using $\epsilon_k$ much smaller than $2^k$ (as $k \geq 2$) and changing the ``tube size'' at $y=1$, this guarantees that near each of these corners the map is biLipschitz and hence BQS there.  
The arclength parameterization guarantees that the map is BQS at 
large scales, and hence BQS.   

After applying the stereographic projection $\varphi$ (projection to the sphere), however, 
there is some $C>0$ such that
for any $i < j$ we have $\diam(\varphi(\cup_{k=i}^j A_{2k}))\ge C$, whereas 
$\diam(\varphi(f(\cup_{k=i}^\infty A_{2k})) \to 0$ as $i \to \infty$. Hence the induced map 
between $\varphi(\Omega)$ and $\varphi(\Omega^\prime)$ is not BQS.

\end{ex}

Now we are ready to study boundary behavior of BQS homeomorphisms between two
unbounded domains.

Let $\Omega,\Omega^\prime$ be two open, connected subsets of metric spaces $X,Y$ respectively, with
$X$ and $Y$ both proper and locally path-connected metric spaces. Let $f:\Omega\to\Omega^\prime$
be a BQS homeomorphism.

In light of Lemma~\ref{lem:bdd-to-bdd}, we now concentrate on the case of both $\Omega$ and $\Omega^\prime$ being
unbounded. So henceforth, we assume that $\Omega$ and $\Omega^\prime$ are unbounded domains in
proper, locally path-connected metric spaces and that $f:\Omega\to\Omega^\prime$ is a BQS homeomorphism.
Recall that we have now to deal with three types of prime ends.

\begin{defn}[{\bf Prime end types}]
We categorize the three types of ends as follows:
\begin{enumerate}
\item[(a)] Ends $[\{E_k\}]$ for which there is some positive integer $k$ with $\overline{E_k}$ compact,
\item[(b)] Ends $[\{E_k\}]$ for which each $E_k$ is unbounded and $I[\{E_k\}]$ is non-empty,
\item[(c)] Ends that are ends at $\infty$; these are also unbounded.
\end{enumerate}
\end{defn}

\begin{remark}
A few words comparing prime ends with respect to the original metric $d_X$ on $\Omega$ and 
prime ends with respect to the induced Mazurkiewicz metric $\dum$ are in order here. 
A prime end of type~(c) with respect to $d_X$ will continue to be a prime end of 
type~(c) with respect to $\dum$. A prime end of type~(b) with respect to $d_X$ will be 
a prime end of type~(c) with respect to $\dum$, and so information about the impression
of such a prime end is lost. A singleton prime end of type~(a) with respect to $d_X$ will
continue to be a singleton prime end of type~(a) with respect to $\dum$. However, a
non-singleton prime end of type~(a) with respect to $d_X$ will cease to be a prime end
(and indeed, cease to be even an end) with respect to $\dum$. Indeed, if $[\{E_k\}]$ is a
prime end of type~(a) with respect to $d_X$ with $I[\{E_k\}]$ non-singleton, then
no point of $I[\{E_k\}]$ can be accessible from inside the chain $\{E_k\}$, as can be
seen from an easy adaptation of the proof of Lemma~\ref{lem:end-div-prime-finite-conn}.
Therefore $\bigcap_k\overline{E_k}^M$ is empty, and so the sequence $\{E_k\}$ fails the 
definition of an end with respect to $\dum$.
\end{remark}

\begin{lemma}\label{lem:end-b-is-prime}
Suppose that $\overline{\Omega}^M$ is proper.
Fix $x_0\in\Omega$ and a strictly increasing sequence $\{n_k\}$ of positive real numbers with $\lim_kn_k=\infty$.
Then, for each $k\in\N$ there is only finitely many unbounded components of $\Omega\setminus\overline{B_M(x_0,n_k)}$.
If for each $k\in\N$ we choose an unbounded component $F_k$ of $\Omega\setminus\overline{B_M(x_0,n_k)}$
such that $F_{k+1}\subset F_k$, then $\{F_k\}$ is a chain for $\Omega$ such that $[\{F_k\}]$ is a prime end of $\Omega$.
\end{lemma}

\begin{proof}
To see the veracity of the first claim, note that if there are infinitely many unbounded components of 
$\Omega\setminus\overline{B_M(x_0,n_k)}$, then there are infinitely many components of $\Omega\setminus\overline{B_M(x_0,n_k)}$
that intersect $B_M(x_0,n_{k+2})\cap\Omega$, in which case we can extract out a sequence of points $\{x_j\}$ from 
$B_M(x_0,n_{k+1})\cap\Omega\setminus B_M(x_0,n_{k+1})$ such that 
for $j\ne i$ we have $2n_{k+2}\ge \dum(x_j,x_i)\ge n_{k+2}-n_{k+1}>0$, and this sequence will not have any subsequence converging
in $\overline{\Omega}^M$, violating the properness of $\overline{\Omega}^M$.

Given the first claim, it is not difficult to see that we can choose a sequence $\{F_k\}$ as in the second hypothesis of the lemma.
Since 
\[
\dM(\Omega\cap\partial F_k,\Omega\cap\partial F_{k+1})\ge 
\dM(\Omega\cap\partial B_M(x_0,n_k),\Omega\cap\partial B_M(x_0,n_{k+1}))\ge n_{k+1}-n_k>0,
\]
it follows that $\{F_k\}$ is a chain of type~(b) if $\bigcap_k\overline{F_k}$ is non-empty, and is a chain of type~(c) (chain at
$\infty$) if $\bigcap_k\overline{F_k}$ is empty. In this latter case, we already have that $[\{F_k\}]$ is prime.
Hence suppose that $\bigcap_k\overline{F_k}$ is non-empty, and suppose that $\{G_k\}\, |\, \{F_k\}$ for some
chain $\{G_k\}$ of $\Omega$. By passing to a subsequence of $\{G_k\}$ if necessary, we can then assume that
$G_k\subset F_k$ for all $k\in\N$. We need to show that $\{F_k\}\, |\, \{G_k\}$. Suppose this is not the case. Then
there is some positive integer $k_0\ge 2$ such that for each $j\in\N$ the set $F_j\setminus G_{k_0}$ is non-empty. Recall that
$G_k\subset F_k$ for each $k\in\N$. Therefore, it follows from the nested property of $\{G_k\}$, that both
$F_j\setminus G_{k_0}$ and $F_j\cap G_{k_0}$ are non-empty for each $j\in\N$. This immediately implies that
$\{G_k\}$ cannot be a chain of type~(a).

Suppose that $\{G_k\}$ is a chain of type~(c).
As $F_j$ is an open connected
subset of $\Omega$, it follows that $\Omega\cap\partial G_{k_0}$ must intersect $F_j$. Hence, for each $j\in\N$ we have
that $\Omega\cap\partial G_{k_0}\cap F_j\ne\emptyset$, which violates the requirement that 
$\diam_M(\Omega\cap\partial G_{k_0})<\infty$. Therefore we have that $\{G_k\}$ cannot be of type~(c).
If $\{G_k\}$ is of type~(b), then as $R_k$ separates $\Omega\cap\partial G_{k_0-1}$ from $\Omega\cap\partial G_{k_0}$,
we must have that for each $j\in\N$, $R_k$ intersects $F_j$. This again violates the requirement that 
$\diam_M(R_k\cap\Omega)<\infty$. 

Therefore we must have that $\{F_k\}\, |\, \{G_k\}$ as well. Hence $[\{F_k\}]$ is a prime end.
\end{proof}

\begin{lemma}\label{lem:bdd-end}
If $\ms{E}$ is a bounded prime end of $\Omega$, then $f(\ms{E})$ is a bounded prime end of $\Omega^\prime$. 
Moreover, if $\ms{E}$ is a singleton prime end (that is, its impression contains only one point), then 
$f(\ms{E})$ is also a singleton prime end.
\end{lemma}

\begin{proof}
Let $\{E_k\}\in\ms{E}$; by passing to a subsequence if need be, we can also assume that 
$\overline{E_1}$ is compact. Moreover, by Lemma~\ref{lem:open-end} 
we can also assume that each $E_k$ is open.
Then the restriction $f\vert_{E_1}$ is also a BQS homeomorphism onto $f(E_1)$, and it follows from
Lemma~\ref{lem:bdd-to-bdd} above that $f(E_1)$ is also bounded. Therefore for each positive integer
$k$ we have that $f(E_k)$ is bounded, and by the properness of $Y$ we also have that
$\overline{f(E_k)}$ is compact. Moreover, since $\overline{E_k}\cap\partial\Omega$ is non-empty,
we can find a sequence $x_m$ of points in $E_k$ converging to some 
$x\in \overline{E_k}\cap\partial\Omega$; since $f$ is a homeomorphism, it follows that
$f(x_m)$ cannot converge to any point in $\Omega^\prime$, and so by the compactness of
$\overline{f(E_k)}$ we have a subsequence converging to a point in 
$\overline{f(E_k)}\cap\partial\Omega^\prime$, that is, 
$\overline{f(E_k)}\cap\partial\Omega^\prime$ is non-empty.
A similar argument also shows that $\bigcap_k\overline{f(E_k)}$ is non-empty and is a 
subset of $\partial\Omega^\prime$.

From Lemma~\ref{lem:maz-dist-bdry-positive} 
we know that for each positive integer $k$,
\[
\dM^\prime(\Omega^\prime\cap\partial f(E_k),\Omega^\prime\cap\partial f(E_{k+1}))>0.
\]
Therefore $f(\ms{E})$ is an end of $\Omega^\prime$.

Next, if $\ms{F}$ is an end that divides $f(\ms{E})$, then with $\{F_k\}\in\ms{F}$ we must have that
$F_k$ is bounded (and hence, $\overline{F_k}$ is compact) for sufficiently large $k$. The above argument,
applied to $\{F_k\}$ and the BQS homeomorphism $f^{-1}$ tells us that $f^{-1}(\ms{F})$ is an end of
$\Omega$ dividing the prime end $\ms{E}$; it follows that $f^{-1}(\ms{F})=\ms{E}$, and so 
$\ms{F}=f(\ms{E})$, that is, $f(\ms{E})$ is a prime end of $\Omega^\prime$.

The last claim of the theorem follows from Lemma~\ref{lem:BQS-single-to-single}.
%
\end{proof}

\begin{lemma}\label{lem:type-b}
Let $\ms{E}$ be a prime end of $\Omega$ of type~(b). Then $f(\ms{E})$ is a prime end of type~(b) for
$\Omega^\prime$ or there is an end $\ms{F}$ at $\infty$ for $\Omega^\prime$ such that
with $\{E_k\}\in\ms{E}$ there exists $\{F_k\}\in\ms{F}$ with $F_{k+1}\subset f(E_k)\subset F_k$
for each positive integer $k$.
\end{lemma} 

\begin{proof}
Here we cannot invoke Lemma~\ref{lem:maz-dist-bdry-positive} as neither $E_k$ nor
$E_{k+1}$ is bounded; but the idea for the proof is similar, as we now show.
Suppose $E_1, E_2\subset\Omega$ be connected open sets with $E_2\subset E_1$ and 
$\tau:=\dM(\Omega\cap\partial E_1,\Omega\cap\partial E_2)>0$. Suppose also that there is a compact
set $R\subset X$ such that $\Omega\cap\partial E_1$ and $\Omega\cap\partial E_2$ belong to different
components of $\Omega\setminus R$ and that $\diam_M(R\cap\Omega)<\infty$. 
We fix a continuum $A\subset\Omega$ that intersects $R$ with $\diam(A)\le 2\diam_M(R\cap\Omega)$.
If $\gamma$ is a continuum in $\Omega$ intersecting both $\Omega\cap\partial E_1$ and 
$\Omega\cap\partial E_2$, then it must intersect $R$. Hence we can find a continuum
$\beta\subset\Omega$ intersecting both $\gamma\cap R$ and $A\cap R$ such that
$\diam_M(\beta)=\diam(\beta)\le 2\diam_M(R\cap\Omega)$. Now by the BQS property of $f$, we see that
\[
\frac{\diam(f(\beta\cup A))}{\diam(f(\gamma))}
\le \eta\left(\frac{\diam(\beta\cup A)}{\diam(\gamma)}\right)
\le \eta\left(\frac{\diam(\beta\cup A)}{\tau}\right).
\]
Note that $\beta$ and $A$ are intersecting continua with diameter at most $2\diam_M(R\cap\Omega)$. It follows
that $\diam(\beta\cup A)\le 4\diam_M(R\cap\Omega)<\infty$. Hence
\[
\frac{\diam(f(A))}{\diam(f(\gamma))}\le \eta\left(\frac{4\diam_M(R\cap\Omega)}{\tau}\right),
\]
that is,
\[
0<\frac{\diam(f(A))}{\eta\left(\frac{4\diam_M(R\cap\Omega)}{\tau}\right)}\le \diam(f(\gamma)).
\]
It follows that 
\[
\dM(\Omega^\prime\cap\partial f(E_1),\Omega^\prime\cap\partial f(E_2))>0.
\]
From the second claim of Lemma~\ref{lem:bdd-to-bdd}, we know that 
$f(R\cap\Omega)$ is bounded with respect to the Mazurkiewicz metric $\dumm$
and hence with respect to $d_Y$. From the fact that 
$f$ is a homeomorphism from $\Omega$ to $\Omega^\prime$,
it also follows that $\Omega^\prime\cap\partial f(E_1)$ and 
$\Omega^\prime\cap\partial f(E_2)$ belong to different components of 
$\Omega^\prime\setminus f(R\cap\Omega)$.

The properness of
$Y$ together with the fact that $\diam_{d_Y}(f(\Omega\cap R))\le \diam_M^\prime(f(\Omega\cap R))<\infty$
tells us that $\overline{f(\Omega\cap R)}$ is compact in $Y$.

From the above argument, we see that $f(\ms{E})$ is an end of $\Omega^\prime$
provided that $\bigcap_k\overline{f(E_k)}$ is non-empty.
Similar argument as at the end of the proof of Lemma~\ref{lem:bdd-end} then tells us that $f(\ms{E})$ is a
prime end. 

Suppose  now that $\bigcap_k\overline{f(E_k)}\ne\emptyset$.
To see that $f(\ms{E})$ is
of type~(b), note that with $\{E_k\}\in\ms{E}$, each $E_k$ is unbounded (and without loss of generality,
open, see Lemma~\ref{lem:open-end}). Hence by Lemma~\ref{lem:bdd-to-bdd} we know that $f(E_k)$ is unbounded.
This completes the proof in the case that $\bigcap_k\overline{f(E_k)}$ is non-empty.

Finally, suppose that $\bigcap_k\overline{f(E_k)}$ is empty. We construct the 
end at $\infty$, $\ms{F}$, as follows. 
By Lemma~\ref{lem:E1-E2-separation}, we can assume that the separating sets $R_k$ have the 
added feature that $E_{k+1}$ is a subset of a component of $\Omega\setminus R_k$
and that that component is a subset of $E_k$. So for each positive integer $k$ we set
$F_k^*$ to be the image of this component under $f$. 
Then $\bigcap_k\overline{F_k}\subset\bigcap_k\overline{f(E_k)}=\emptyset$.
Moreover, $F_k^*$ is a component of $\Omega^\prime\setminus K_k$ with $K_k=\overline{f(R_k\cap\Omega)}$, 
and as $R_k$ is compact and $\diam_M(R_k\cap\Omega)<\infty$, by Lemma~\ref{lem:bdd-to-bdd} we know that
$\diam_M^\prime(f(R_k\cap\Omega))<\infty$ and so $\overline{f(R_k\cap\Omega)}$ is compact in $Y$.
We cannot however guarantee that 
$\dM^\prime(\Omega^\prime\cap \partial F_k^*,\Omega^\prime\cap\partial F_{k+1}^*)$ is positive; hence
we set $F_k:=F_{2k-1}^*$ and note that
\[
\dM^\prime(\Omega^\prime\cap \partial F_k,\Omega^\prime\cap\partial F_{k+1})
\ge \dM^\prime(\Omega^\prime\cap\partial E_{2k},\Omega^\prime\cap\partial E_{2k+1})>0.
\]
Thus the sequence $\{F_k\}$ is an end at $\infty$ of $\Omega^\prime$, completing the proof.
\end{proof}

Given that ends at infinity are prime ends (no other ends divide them, see~\cite{E}), it follows that should
$\ms{E}$ be associated with an end $\ms{F}$ at $\infty$, there can be only one such end at $\infty$ it can
be associated with. 

We next consider the last of the three types of prime ends of $\Omega$. We say that a chain $\{E_k\}$
in an end $\ms{E}$ at $\infty$ of $\Omega$ 
is a \emph{standard representative} of $\ms{E}$ if there is a point $x_0\in X$ and a strictly monotone 
increasing sequence $R_k\to\infty$ such that for each $k$ the set $E_k$ is a connected component
of $\Omega\setminus \overline{B_M(x_0,R_k)}$. Note that if $H\subset\Omega$ such that 
$\diam_M(H)<\infty$, then $diam_M(\overline{H}\cap\Omega)$ is also finite.

\begin{lemma}\label{lem:standard-rep}
Let $\ms{E}$ be an end at $\infty$ of $\Omega$.  Then there is a standard representative 
$\{E_k\}\in\ms{E}$. Moreover, for each $x_0\in\Omega$
there is a representative chain $\{F_k\}\in\ms{E}$ such that
for each $k\in\N$, $F_k$ is a component of $\Omega\setminus\overline{B_M(x_0,k)}$. 
\end{lemma}

\begin{proof}
Let $\{G_k\}\in\ms{E}$ and for each $k$ let $K_k$ be a compact subset of $X$ such that $\diam_M(K_k\cap\Omega)$
is finite and $G_k$ is an unbounded component of $\Omega\setminus K_k$. Fix $x_0\in\Omega$. Since $\bigcap\overline{E_k}$
is empty, we can assume without loss of generality that $x_0\not\in\overline{E_1}$. Then $x_0$ is in a different component
from $E_k$ of the set $\Omega\setminus K_k$ for each $k\in\N$. In what follows, by $B_M(x_0,\tau)$ we mean
the ball \emph{in $\Omega$} centered at $x_0$ with radius $\tau$ with respect to the Mazurkiewicz metric $\dum$.

As $\tau_k:=\diam_M(K_k)<\infty$, inductively we can find a sequence of strictly monotone increasing
positive real numbers $R_k$ as follows. Let 
\[
R_1:=\max\{1, \diam_M(\{x_0\}\cup K_1)\}. 
\]
We claim that
there is some $L_1\in\N$ such that $\overline{G_{1+L_1}}$ does not intersect $\overline{B_M(x_0,R_1)}$. If this claim is 
wrong, then for each positive integer $j\ge 2$ there is a point $x_j\in\overline{G_j}\cap \overline{B_M(x_0,R_1)}$,
in which case, by the compactness of the set $\overline{B_M(x_0,R_1)}$ there is a subsequence $x_{j_n}$ and
a point $x_\infty\in \overline{B_M(x_0,R_1)}\subset X$ such that $x_{j_n}\to x_\infty$; note by the nestedness property
of the chain $\{G_k\}$ that $x_\infty\in\overline{G_k}$ for each $k$, violating the property that $\bigcap_k\overline{G_k}$
is empty. Therefore the claim holds. We set $E_1$ to be the connected component of
$\Omega\setminus \overline{B_M(x_0,R_1)}$ that contains $G_{1+L_1}$. Set $n_1=1+L_1$. We then set
\[
R_2:=\max\{2, R_1+1, \diam_M(\{x_0\}\cup K_{n_1})\},
\]
and as before, note that there is some $n_2>n_1$ such that $\overline{G_{n_2}}$ is disjoint from
$\overline{B_M(x_0,R_2)}$. We set $E_2$ to be the component of $\Omega\setminus\overline{B_M(x_0,R_2)}$
that contains $G_{n_2}$. Thus inductively, once $R_1,\cdots, R_k$ have been determined and 
the corresponding strictly monotone increasing sequence of integers $n_1,\cdots, n_k$ have also been chosen,
we set
\[
R_{k+1}:=\max\{k+1,R_k+1, \diam_M(\{x_0\}\cup K_{n_k})\}
\]
and choose $n_{k+1}$ such that $n_{k+1}>n_k$ and $\overline{G_{n_{k+1}}}\cap \overline{B_M(x_0,R_{k+1})}$
is empty, and set $E_{k+1}$ to be the component of $\Omega\setminus \overline{B_M(x_0,R_{k+1})}$ that contains
$G_{n_{k+1}}$. Thus we obtain the sequence $\{E_k\}$; it is not difficult to see that this sequence is a chain at $\infty$ 
according to Definition~\ref{def:type-c-ends}. Indeed, as $G_{n_{k+1}}\subset E_{k+1}$, and as
$G_{n_{k+1}}\subset G_{n_k}$ and $K_{n_k}\cap E_{k+1}=\emptyset$, it follows that $E_{k+1}$ belongs to 
the same component as $G_{n_{k+1}}$ of $\Omega\setminus K_{n_k}$, and so $E_{k+1}\subset G_{n_k}$. Thus we have
\[
G_{n_{k+1}}\subset E_{k+1}\subset G_{n_k}\subset E_k,
\]
and so we have both the nested property of the sequence $\{E_k\}$ and the fact that $\{G_k\}$ divides $\{E_k\}$.
Therefore,  $[\{E_k\}]$ is an end at $\infty$ of $\Omega$ and hence is a prime end, and thus $[\{E_k\}]=\ms{E}$; see~\cite[Remark~4.1.11, Remark~4.1.14]{E} for details of the proof.

The last part of the claim follows from knowing that for each $k\in\N$ there are positive integers
$j,l$ such that $R_k\le j\le R_{k+l}$.
\end{proof}

\begin{lemma}
Let $\ms{E}$ be an end at $\infty$ of $\Omega$. Then either $\bigcap_k\overline{f(E_k)}=\emptyset$ for each
$\{E_k\}\in\ms{E}$ or $\bigcap_k\overline{f(E_k)}$ non-empty for each $\{E_k\}\in\ms{E}$. In the first case
$f(\ms{E})$ is an end at $\infty$ of $\Omega^\prime$. In the second case $f(\ms{E})$ is a 
prime end of type~(b) of $\Omega^\prime$.
\end{lemma}

\begin{proof}
Let $\{E_k\}\in\ms{E}$ be a standard representative; hence, there is some $x_0\in \Omega$
and for each $k\in\N$ there exists $R_k>0$ such that
the sequence $\{R_k\}$ is strictly monotone increasing with $\lim_k R_k=\infty$, and $E_k$ is a component
of $\Omega\setminus \overline{B_M(x_0,R_k)}$. 
Fix a continuum $A\subset\Omega$ intersecting $\Omega\cap\partial E_k$ and with
\[
\diam(A)\le 2\diam_M(\Omega\cap\partial E_k)<\infty.
\]

Let $z\in\Omega^\prime\cap\partial f(E_k)$ and
$w\in\Omega^\prime\cap\partial f(E_{k+1})$, and let $\gamma$ be a continuum in $\Omega^\prime$ containing
$z,w$. Then $f^{-1}(\gamma)$ is a continuum in $\Omega$ containing
$f^{-1}(z)\in \Omega\cap\partial E_k$ and $f^{-1}(w)\in\Omega\cap\partial E_{k+1}$. By the fact that
$\{E_k\}$ is a chain of $\Omega$, we know that 
\[
\diam(f^{-1}(\gamma))\ge \tau:=\dist_M(\Omega\cap\partial E_k, \Omega\cap\partial E_{k+1})>0.
\]
As $f^{-1}(\gamma)$ intersects both $\Omega\cap\partial E_k$ and $\Omega\cap\partial E_{k+1}$,
we can find a continuum $\beta$ connecting $A\cap\partial E_k$ and $f^{-1}(\gamma)\cap\partial E_k$
such that $\diam(\beta)\le 2\diam_M(\Omega\cap\partial E_k)$. By the BQS property, we now have
\[
\frac{\diam(f(\beta\cup A))}{\diam(\gamma)}\le \eta\left(\frac{\diam(\beta\cup A)}{\diam(f^{-1}(\gamma)}\right)
\le \eta\left(\frac{4\diam_M(\Omega\cap\partial E_k)}{\tau}\right),
\]
and so
\[
0<\frac{\diam(f(A))}{\eta\left(\frac{4\diam_M(\Omega\cap\partial E_k)}{\tau}\right)}\le \diam(\gamma),
\]
and so we have that
\begin{equation}\label{eq:positive-separation1}
\dist_M(\Omega^\prime\cap\partial f(E_k),\Omega^\prime\cap\partial f(E_{k+1}))
 \ge \frac{\diam(f(A))}{\eta\left(\frac{4\diam_M(\Omega\cap\partial E_k)}{\tau}\right)}>0.
\end{equation}
Note that as neither $E_k$ nor $E_{k+1}$ is bounded, we cannot utilize that lemma directly,
nor can we directly call upon the relevant part of the proof of Lemma~\ref{lem:type-b} as there is no
separating set $R_k$; but the idea of the proof is quite similar. For the convenience of the reader
we gave the complete proof above.

Now we have two possibilities, either $\bigcap_k\overline{f(E_k)}$ is empty, or it is not empty.

\noindent{\bf Case 1:} $\bigcap_k\overline{f(E_k)}=\emptyset$. For each $k\in\N$ note that
$f(E_k)$ is a component of $\Omega^\prime\setminus f(\overline{B_M(x_0,R_k)}\cap\Omega)$.
By Lemma~\ref{lem:bdd-to-bdd}, we know that 
$\dist_M^\prime(f(\overline{B_M(x_0,R_k)}\cap\Omega))<\infty$, and moreover, as the topology
on $\Omega$ with respect to $d_X$ and with respect to $\dum$ are the same, we also have that
\[
\overline{f(\overline{B_M(x_0,R_k)}\cap\Omega)}\cap\Omega^\prime=f(\overline{B_M(x_0,R_k)}\cap\Omega).
\]
Also, $\overline{f(\overline{B_M(x_0,R_k)}\cap\Omega)}$ is compact as $Y$ is proper
and $\diam(\overline{f(\overline{B_M(x_0,R_k)}\cap\Omega)})=\diam(f(\overline{B_M(x_0,R_k)}\cap\Omega)
\le \diam_M(f(\overline{B_M(x_0,R_k)}\cap\Omega)<\infty$. Again, as $Y$ is proper, $\overline{f(E_k)}$
is also proper.
Therefore $f(E_k)$ is an acceptable set at $\infty$, and by~\eqref{eq:positive-separation1} we see
that $\{f(E_k)\}$ is a chain at $\infty$ of $\Omega^\prime$.

\noindent{\bf Case 2:} $\bigcap_k\overline{f(E_k)}\ne\emptyset$. In this case we need to show that
$\{f(E_k)\}$ forms a chain of type~(b). The only additional condition we need to check is that
for each $k\in\N$ there is some compact set $P_k\subset Y$ such that 
$\Omega^\prime\cap\partial f(E_k)$ and $\Omega^\prime\cap\partial f(E_{k+1})$ are in different
\emph{components} of $\Omega^\prime\setminus P_k$. 
In what follows, by $S_M(x_0,\tau)$ we mean the set $\{y\in\Omega\, :\, \dum(x_0,y)=\tau\}$.
We choose 
\[
P_k=\overline{f(S_M(x_0,\tfrac{R_k+R_{k+1}}{2}))}.
\]
Since $\Omega\cap\partial E_k\subset\Omega\cap\partial B_M(x_0,R_k)$ and
$\Omega\cap\partial E_{k+1}\subset\Omega\cap\partial B_M(x_0,R_{k+1})$, it follows that these
two sets cannot intersect the same connected component of 
$\Omega\setminus S_M(x_0,\tfrac{R_k+R_{k+1}}{2}))$. Since 
$\overline{B_M(x_0,R_k)}\cap\Omega$ is connected in $\Omega$, it follows that
$\Omega\cap\partial E_k$ is in one component of $\Omega\setminus S_M(x_0,\tfrac{R_k+R_{k+1}}{2}))$.
As $\Omega\cap\overline{E_{k+1}}\subset\Omega\setminus B_M(x_0,R_{k+1})$ is connected,
it also follows that $\Omega\cap\partial E_{k+1}$ lies in a component of 
$\Omega\setminus S_M(x_0,\tfrac{R_k+R_{k+1}}{2}))$. As $f$ is a homeomorphism, similar statements
hold for $\Omega^\prime\cap\partial f(E_k)$, $\Omega^\prime\cap\partial f(E_{k+1})$,
that is, they are contained in two different components of $\Omega\setminus P_k$.
As $\diam_M(P_k\cap\Omega^\prime)<\infty$ by Lemma~\ref{lem:bdd-to-bdd}, it follows also that
$P_k$ is a compact subset of $Y$. Thus the separation condition for the sequence 
$\{f(E_k)\}$ is verified, and so this sequence is a chain of $\Omega^\prime$ of type~(b).

To see that $[\{f(E_k)\}]$ is a prime end, note that if $\{G_k\}$ is a chain that divides
$\{f(E_k)\}$, then $\{f^{-1}(G_k)\}$ is a chain of $\Omega$ that divides $\{E_k\}$; as 
$\{E_k\}$ is a prime end, it follows that $\{E_k\}$ divides $\{f^{-1}(G_k)\}$, and so
$\{f(E_k)\}$ divides $\{G_k\}$ as well. Thus the corresponding end is a prime end.
\end{proof}

\begin{ex}
In Example~\ref{ex:not-BQS-Plane-Sphere}, $\Omega$ has a prime end of type~(b) 
with impression $\{0\} \times [0,\infty)$ (e.g. take 
$E_k = \biggl((0, \tfrac{3}{2^k}) \times (0, \infty) \biggr) \cap \Omega$) 
which gets mapped to an end at $\infty$ by $f$.  The BQS map
$f^{-1}$ then also maps an end at infinity to an end of type (b).  
\end{ex}

We now summarize the above study. Recall that we assume $X$ and $Y$ to be locally path-connected
proper metric spaces and that $\Omega\subset X$, $\Omega^\prime\subset Y$ are domains.

\begin{thm}\label{thm:BQS-unbdd}
Suppose that $\Omega$ is an unbounded domain and let $f:\Omega\to\Omega^\prime$ be a BQS homeomorphism.
Then there is a homeomorphism $f_P:\overline{\Omega}^P\to\overline{\Omega^\prime}^P$ such that
$f(\ms{E})$ is a bounded prime end (i.e.~a prime end of type~(a)) if $\ms{E}$ is a bounded prime end, and
$f(\ms{E})$ is either an unbounded prime end (i.e.~of type~(b)) or is an end at $\infty$ of $\Omega^\prime$
if $\ms{E}$ is an unbounded prime end or an end at $\infty$ of $\Omega$. Moreover, if $\ms{E}$ is a singleton
prime end, then so is $f(\ms{E})$.
\end{thm}

\begin{remark}
Combining the above theorem with~\cite[Proposition~10.11]{Hei} we see that if $\Omega$ is finitely connected at the
boundary, then $f_P\vert_{\overline{\Omega}^P\setminus\partial_\infty\Omega}$ is a BQS homeomorphism
with respect to the Mazurkiewicz metric $\dum$,
where $\partial_\infty\Omega$ is the collection of all ends at $\infty$ of $\Omega$. However, unlike in the case
of bounded domains (see Proposition~\ref{prop:bdd-domain-ends-match}(c)), we cannot conclude here that
$\Omega^\prime$ must also be finitely connected at the boundary, see Example \ref{ex:not-BQS-Plane-Sphere};
however, this is not an issue as no continuum in $\overline{\Omega}^P$ contains points from $\partial_\infty\Omega$.
The obstacle however is to know about $f^{-1}_P$, and therefore it is natural to ask whether there are continua in 
$\overline{\Omega^\prime}^P$ that do not lie entirely in $\Omega^\prime\cup\partial_B\Omega^\prime$,
where $\partial_B\Omega^\prime$ is the collection of all bounded prime ends of $\Omega^\prime$ (which is
the same as the collection of all singleton prime ends, thanks to the above theorem). Unfortunately it is possible for 
$\partial_P\Omega^\prime\setminus[\partial_B\Omega^\prime\cup\partial_\infty\Omega^\prime]$ to contain many continua, as
Example~\ref{ex:infinite-dog-poo-volume} below shows. We do not have a notion of metric on $\overline{\Omega^\prime}^P$
in this case, and so it does not make sense to ask that $f_P$ is itself a BQS map unless $\Omega^\prime$ is also
finitely connected at the boundary (in which case, it is indeed a BQS homeomorphism).
\end{remark}

\begin{ex}\label{ex:infinite-dog-poo-volume}
Let $\Omega_0$ be the planar domain 
\[
\{(x,y)\in\R^2\, :\, x>0, 0<y<\tfrac{1}{1+x}\}\setminus \left(
\bigcup_{2\le n\in\N} [0,2n-2]\times\{\tfrac{1}{2n}\}
\bigcup_{n\in\N}[\tfrac{1}{2}, 2n]\times\{\tfrac{1}{2n+1}\}\right),
\]
and set $\Omega=\Omega_0\times(0,\infty)$. 

\begin{figure}[!htb]
\begin{center}
\begin{tikzpicture}[xscale = 1, yscale=5, domain=0:8,samples=600, smooth]
\draw[black, thick] (0,0) -- (0, .6);
\draw[black, thick] (0,0) -- (8, 0);
\draw[black, domain=.666:8] plot(\x, {1/(\x+1)});
\draw[black] (0, .25) -- (2, .25);
\draw[black] (.5, .2) -- (4, .2);
\draw[black] (0, .166) -- (4, .166);
\draw[black] (.5, .143) -- (6, .143);
\draw[black] (0, .125) -- (6, .125);
\draw[black] (.5, .111) -- (8, .111);
\node at (4, .08) {$\vdots$};
\end{tikzpicture}
\end{center}
\caption{$\Omega_0$}
\end{figure}

Note that $\Omega$ is not finitely connected at the boundary.
For each $t>0$ and $k\in\N$ consider
\[
E_k^t:=\{(x,y,z)\in\R^3\, :\, y^2+(z-t)^2<\tfrac{1}{(k+1)^2}\}.
\]
Observe that for each $t>0$ the sequence $\{E_k^t\}$ is a chain of $\Omega$ with impression
$\R\times\{0\}\times\{t\}$. The corresponding end is a prime end, and for each
compact interval $I\subset(0,\infty)$, the set
$\{[\{E_k^t\}]\, :\, t\in I\}$ is a subset of $\partial_P\Omega\setminus[\partial_B\Omega\cup\partial_\infty\Omega]$
and is a continuum. Note that in this example, $\partial_\infty\Omega$ is empty. 
By opening up the domain along the slitting planes $[0,2n-2]\times\{\tfrac{1}{2n}\}$ and 
$[\tfrac{1}{2}, 2n]\times\{\tfrac{1}{2n+1}\}$, we obtain a domain that is finitely connected at the boundary, and
the map from $\Omega$ to this domain can be seen to be a BQS homeomorphism.
\end{ex}

The following example shows that it is possible to have a prime end in the sense of~\cite{E} which 
makes the prime end closure sequentially compact, but no equivalent prime end in our sense.

\begin{ex}
We construct $\Omega$ by removing closed subsets of $\R^2$.  
The ``barrier" $G(h)$ centered at $(0,0)$ with height $h$ is given by
$G(h) = \bigcup_{k \in \Z} C_k$, where $C_0:=((-\infty,-1]\cup[1,\infty))\times\{0\}$ and for $h, k > 0$,
\begin{equation*}
\begin{split}
C_k &= (\{1 - 2^{-k}\} \times [-h2^{-k}, h2^{-k}]) \cup ([1-2^{-k}, 2^k] \times \{-h2^{-k}, h2^{-k}\}) \\
C_{-k} &= (\{-1 + 2^{-k}\} \times [-h2^{-k}, h2^{-k}]) \cup ([-2^k, -1+2^{-k}] \times \{-h2^{-k}, h2^{-k}\})
\end{split}
\end{equation*}
Thus $C_{-k}$ is obtained from $C_k$ by reflecting it about the $y$-axis. See Figure~3 below for an illustration 
of the barrier $G(h)$.

\begin{figure}[!htb]
\begin{center}
\begin{tikzpicture}
\draw[black, thick] (-6, 0) -- (-1, 0);
\draw[black, thick] (1,0) -- (6,0);

\draw[<->, black, thick] (0, -.5) -- (0, .5);
\node at (-.25,0) {$h$};

\draw[black] (.5, -.5) -- (.5, .5);
\draw[black] (.5, .5) -- (2, .5);
\draw[black] (.5, -.5) -- (2, -.5);

\draw[black] (.75, -.25) -- (.75, .25);
\draw[black] (.75, .25) -- (4, .25);
\draw[black] (.75, -.25) -- (4, -.25);

\draw[black] (.875, -.125) -- (.875, .125);
\draw[black] (.875, .125) -- (6, .125);
\draw[black] (.875, -.125) -- (6, -.125);

\draw[black] (-.5, .5) -- (-.5, -.5);
\draw[black] (-.5, -.5) -- (-2, -.5);
\draw[black] (-.5, .5) -- (-2, .5);

\draw[black] (-.75, .25) -- (-.75, -.25);
\draw[black] (-.75, -.25) -- (-4, -.25);
\draw[black] (-.75, .25) -- (-4, .25);

\draw[black] (-.875, .125) -- (-.875, -.125);
\draw[black] (-.875, -.125) -- (-6, -.125);
\draw[black] (-.875, .125) -- (-6, .125);

\end{tikzpicture}
\end{center}
\caption{The gate $G(h)$ centered at $(0,0)$.}
\end{figure}
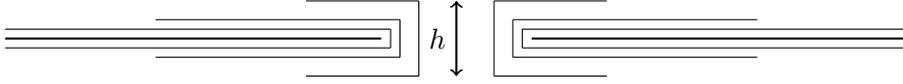

Given a set $A\subset \R^2$ and a point $(a,b)\in\R^2$, we denote by $A+(a,b)$ the set
\[
A+(a,b):=\{(x+a,y+b)\, :\, (x,y)\in A\}.
\]
We now set
\[
\Omega:=\R\times(0,\infty)\setminus \bigcup_{2\le j\in\N} \left(G(2^{-4j})+(2^{j},2^{-j})\right).
\]


%
%
%
%
%
%
%
%
%

The sets $E_j:=\Omega\cap(\R\times(0,2^{-j}))$ forms a chain $\{E_j\}$ of $\Omega$ in the sense of~\cite{E} with
$[\{E_j\}]$ prime, but this does not form a chain in our sense, nor is it equivalent to a chain that is also a chain in our sense.
\end{ex}

\section{Geometric quasiconformal maps and prime ends}

In this section we consider a Carath\'eodory extension theorem for geometric quasiconformal maps.
Unlike with BQS homeomorphisms, geometric quasiconformal maps can map bounded domains to
unbounded domains. In this section we will consider quasiconformal maps between two
domains without assuming boundedness or unboundedness for either domain. As we have
less control with geometric quasiconformal maps than with BQS maps, here we need additional
constraints on the two domains. The main theorem of this section is Theorem~\ref{thm:main-geomQC}.

In this section, we will assume that $X$ and $Y$ are proper metric spaces that are locally path-connected,
and that $\Omega\subset X$, $\Omega^\prime\subset Y$ are domains. We also assume that 
both $(\Omega,\dum)$ and
$(\Omega^\prime,\dumm)$ are Ahlfors $Q$-regular and support a $Q$-Poincar\'e inequality 
for some $Q>1$,
$\mu$ is a Radon measure on $\Omega$ such that bounded subsets of $\Omega$ have
finite measure, 
and that $f:\Omega\to\Omega^\prime$ is a geometric quasiconformal mapping. In light of
the following lemma, if both $\Omega$ and $\Omega^\prime$ are bounded, then 
their prime end closures and their Mazurkiewicz metric completions are equivalent,
in which case by the results of~\cite{HeiK} and~\cite[Proposition~10.11]{Hei}, 
we know that geometric quasiconformal
maps between them have even a geometric
quasiconformal homeomorphic extension to their prime end closures. Thus the interesting
situation to consider here is when at least one of the domains is unbounded. 

Before addressing the extension result given in Theorem~\ref{thm:main-geomQC} below,
we first need to study the properties of prime ends of domains that are Ahlfors regular and support
a Poincar\'e inequality as above.

\begin{lemma}\label{lem:finite-conn-at-bdy-implication}
The domain $\Omega$ is finitely connected at the boundary if and only if the following three
conditions hold:
\begin{enumerate}
  \item[(i)] $\overline{\Omega}^M$ is proper,
  \item[(ii)] $\Omega$ has no type~(b) prime ends, 
  \item[(iii)] $\overline{\Omega}^P$ is sequentially compact.
\end{enumerate}
\end{lemma}

\begin{proof}
We first show that $\overline{\Omega}^M$ is proper. It suffices to show that bounded sequences in
$\Omega$ have a convergent subsequence converging to a point in $\overline{\Omega}^M$.  
Let $\{x_n\}$ be a bounded sequence (with respect  to the 
Mazurkiewicz metric $\dum$) in $\Omega$. Then it is a bounded sequence in $X$, and as $X$ is proper,
it has a subsequence, also denoted $\{x_n\}$, and a point $x_0\in X$ such that 
$\lim_nx_n=x_0$, the limit occurring with respect to the metric $d_X$. If $x_0\in\Omega$, then by
the local path-connectivity of $X$, we also have that the limit occurs with respect to $\dum$ as well.
If $x_0\not\in\Omega$, then $x_0\in\partial \Omega$. Since $\Omega$ is finitely connected at the boundary,
it follows that there are only finitely many components $U_1(1),\cdots, U_{k_1}(1)$ of $B(x_0,1)\cap\Omega$
with $x_0\in\partial U_j(1)$ for $j=1,\cdots, k_1$, and we also have that 
$B(x_0,\rho_1)\cap\Omega\subset\bigcup_{j=1}^{k_1}U_j(1)$ for some $\rho_1>0$. We have a choice of some
$j_1\in\{1,\cdots,k_1\}$ such that infinitely many of the terms in the sequence $\{x_n\}$ lie in
$U_{j_1}(1)$. this gives us a subsequence $\{x_n^1\}$ with $\dum(x_n^1,x_m^1)\le 1$. 
Similarly, we have finitely many components $U_1(2),\cdots, U_{k_2}(2)$ of $B(x_0,2^{-1})\cap\Omega$,
with $x_0\in\partial U_j(2)$ for $j=1,\cdots, k_2$, and we also have that 
$B(x_0,\rho_2)\cap\Omega\subset\bigcup_{j=1}^{k_2}U_j(2)$ for some $\rho_2>0$.
We then find a further subsequence $\{x_n^2\}$ of the sequence $\{x_n^1\}$ lying entirely in
$U_{j_2}(2)$ for some $j_2\in\{1,\cdots,k_2\}$. Note that $\dum(x_n^2,x_m^2)\le 2^{-1}$.
Inductively, for each $l\in\N$ with $l\ge 2$ 
we can find a subsequence $\{x_n^l\}$ contained within a component
of $B(x_0,2^{-l})\cap\Omega$ with that component containing $x_0$ in its boundary, such that
$\{x_n^l\}$ is a subsequence of $\{x_n^{l-1}\}$. We then have $\dum(x_n^l,x_m^l)\le 2^{-l}$.
Now a Cantor diagonalization argument gives a subsequence $\{y_l\}$ of the original sequence $\{x_n\}$
such that for each $l\in\N$, $\dum(y_l,y_{l+1})\le 2^{-l}$, that is, this subsequence is Cauchy with
respect to the Mazurkiewicz metric $\dum$ and so is convergent to a point in $\overline{\Omega}^M$.

We next verify that there cannot be any prime end of type~(b). Suppose $\ms{E}$ is a prime end of
$\Omega$ with impression $I(\ms{E})$ containing more than one point. Then by 
Lemma~\ref{lem:end-div-prime-finite-conn} there is a singleton prime end $\ms{G}$ with
$\ms{G}\, |\, \ms{E}$, and as the impression of $\ms{E}$ is not a singleton set, we have
$\ms{G}\ne\ms{E}$, violating the primality of $\ms{E}$. Thus there cannot be a prime end of type~(b),
nor a non-singleton prime end of type~(a).

To verify that $\overline{\Omega}^P$ is sequentially compact, we first consider sequences
$\{x_k\}$ in $\Omega$. If this sequence is bounded in the metric $d_X$, then
by the properness of $X$ there is a subsequence $\{x_{n_j}\}$ that converges to a
point $x_0\in\overline{\Omega}$. If $x_0\in\Omega$, then that subsequence converges to $x_0$ also
in $\overline{\Omega}^P$. If $x_0\in\partial\Omega$, then using the finite connectivity of
$\Omega$ at $x_0$ we can find a singleton prime end $\ms{E}=[\{E_k\}]$ with $\{x_0\}=I(\ms{E})$
such that for each $k\in\N$ there is an infinite number of positive integers $j$ for which 
$x_{n_j}\in E_k$. Now a Cantor-type diagonalization argument gives a further subsequence that converges
in $\overline{\Omega}^P$ to $\ms{E}$. If the sequence is not bounded in $d_X$, then we argue as follows. 
Fix $z\in\Omega$. Then
for each $n\in\N$ there are only finitely many components of $\Omega\setminus B_M(z,n)$ that intersect
$\Omega\setminus B_M(z,n+1)$ (for if not, then we can find a sequence of points $y_j$ in the unbounded
components of $\Omega\setminus B(z,n)$,  with no two belonging to the same unbounded
component, such that $\dum(z,y_j)=n+\tfrac12$ and $\dum(y_j,y_l)\ge \tfrac12$ for $j\ne l$; then $\{y_j\}$ would be
a bounded sequence with respect to $d_X$, and the properness of $X$ gives us a subsequence 
that converges to some $y_0\in\partial\Omega$, and $\Omega$ would then not be finitely connected at $y_0$).
Hence there is an unbounded component $U_1$ of $\Omega\setminus B_M(z,1)$ that contains $x_k$ for
infinitely many $k\in\N$, that is, there is a subsequence $\{x^1_j\}$ of $\{x_k\}$ that is contained in $U_1$.
There is an unbounded component $U_2$ of $\Omega\setminus B_M(z,2)$ that contains
$x^1_j$ for infinitely many $j\in\N$, and so we obtain a further subsequence $\{x^2_j\}$ lying in $U_2$.
Since this further subsequence also lies in $U_1$, it follows that $U_2\subset U_1$. Thus inductively we
find unbounded components $U_n$ of $\Omega\setminus B_M(z,n)$ and subsequence $\{x^n_j\}$ of
$\{x^{n-1}_j\}$ such that $x^n_j\in U_n$; moreover, $U_n\subset U_{n-1}$. If $\bigcap_n\overline{U_n}$
is non-empty, then $\bigcap_n\overline{U_n}\subset\partial\Omega$ (because for every point $w\in\Omega$
we know from the local connectivity of $X$ that $\dum(z,w)<\infty$), and $\Omega$ will not be 
finitely connected at each point in $\bigcap_n\overline{U_n}$. It follows that we must have
$\bigcap_n\overline{U_n}=\emptyset$. Therefore $\{U_n\}$ is an end at $\infty$ for $\Omega$, and 
by considering the subsequence $z_n=x^n_n$ of the original sequence $\{x_k\}$, we see that
$\{z_n\}$ converges in $\overline{\Omega}^P$ to $\{U_n\}$. This concludes the proof that
sequences in $\Omega$ have a subsequence that converges in $\overline{\Omega}^P$.

Next, if we have a sequence $\ms{E}^n$ of points in $\partial_P\Omega$ that are all singleton
prime ends (that is, they have a representative in $\partial_M\Omega$), then we  can use
the Mazurkiewicz metric to choose $\{E^n_k\}\in\ms{E}^n$ such that for each $n,k\in\N$ we have
$\diam_M(E^n_k)<2^{-(k+n)}$, and choose $y_n=x^n_n\in E^n_n$ to obtain a sequence in $\Omega$.
The above argument then gives a subsequence $y_{n_j}$ and $\ms{F}\in\partial_P\Omega$ such that
$y_{n_j}\to\ms{F}$ as $j\to\infty$ in the prime end topology. If $\ms{F}$ is a singleton prime 
end, then the properness of $\overline{\Omega}^M$ implies that $\ms{E}^{n_j}\to\ms{F}$.
If $\ms{F}$ is not a singleton prime end, then it is an end at $\infty$ (because we have shown
that there are no ends of type~(b)), in which case we can find a sequence of positive
real numbers $R_k\to\infty$ and $\{F_k\}\in\ms{F}$ such that$F_k$ is a component of
$\Omega\setminus \overline{B_M(x_0,R_k)}$ for some fixed $x_0\in\Omega$, see 
Lemma~\ref{lem:standard-rep}. Since we can find $L\in\N$ such that $R_{k+L}-R_k>1$, and as
$y_j\in F_{k+L}$ for sufficiently large $j\in\N$, it follows that $E^{n_j}_{n_j}\subset F_k$.
Therefore $\ms{E}^{n_j}\to\ms{F}$. Finally, if $\ms{E}^n$ are all ends at $\infty$, then 
from the second part of Lemma~\ref{lem:standard-rep} we can fix $x_0\in\Omega$ and
choose $\{E^n_k\}\in\ms{E}^n$ with $E^n_k$ a component of $\Omega\setminus\overline{B_M(x_0,k)}$.
For each $k\in\N$ there are only finitely many components of $\Omega\setminus\overline{B_M(x_0,k)}$
that are unbounded, and so it follows that there is some component $F_1$ of 
$\Omega\setminus\overline{B_M(x_0,1)}$ for which the set
\[
I(1):=\{n\in\N\, :\, E^n_1=F_1\}
\]
has infinitely many terms. 
Note that necessarily $F_2\subset F_1$.
It then follows that there is some component $F_2$ of 
$\Omega\setminus\overline{B_M(x_0,2)}$ such that
\[
I(2):=\{n\in I(1)\, :\, n\ge 2\text{ and }E^n_2=F_2\}
\]
has infinitely many terms. Proceeding inductively, we obtain sets $I(j+1)\subset I(j)$ of infinite
cardinality and $F_j$ such that $\{F_j\}$ is an end at $\infty$, and a sequence $n_j\in\N$
with $j\le n_j\in I(j)$ such that $E^{n_j}_j=F_j$. It follows that $\ms{E}^n$ converges in the 
prime end topology to $\{F_j\}$.

Now we prove that if $\Omega$ satisfies conditions~(i)--(iii), then it is finitely connected at the boundary. Recall the definition of finitely connected at the boundary from 
Definition~\ref{defn:finiteConnBdy}.
Suppose that $x_0\in\partial\Omega$ such that $\Omega$ is not finitely connected at $x_0$. 
Then there is some $r>0$ such that either there are infinitely many components of $B(x_0,r)\cap\Omega$
with $x_0$ in their boundaries, or else there are only finitely many such components $U_1,\cdots, U_k$
such that for all $\rho>0$, $B(x_0,\rho)\setminus\bigcup_{j=1}^kU_j$ is non-empty.
In the first case, $B(x_0,r/2)$ intersects each of those components, and hence infinitely
many components of $B(x_0,r)\cap\Omega$ intersect $B(x_0,r/2)$. In the second case,
we see that $B(x_0,r/2)\setminus \bigcup_{j=1}^kU_j$ is non-empty; in this case, if there are
only finitely many components of $B(x_0,r)\cap\Omega$ that intersect $B(x_0,r/2)$, then 
for each of these components $W$ that are not one of $U_1,\cdots, U_k$, we must have that 
$\dist(x_0,W)>0$, and so the choice of $\rho\le r/2$ as no larger than the minimum of $\dist(x_0,W)$ over
all such $W$ tells us that
$B(x_0,\rho)\cap\Omega\subset\bigcup_{j=1}^kU_j$, contradicting the fact that 
$\Omega$ is not finitely connected at $x_0$. 

From the above two cases, we know that  
there is some $r>0$ such that there are infinitely many components of $B(x_0,r)\cap\Omega$
intersecting $B(x_0,r/2)$. 
Hence we can choose a sequence of points, $y_j\in\Omega$, with no two
belonging to the same component of $B(x_0,r)\cap\Omega$, such that $d(x_0,y_j)=r/2$. Then
we have that for each $j,l\in\N$ with $j\ne l$, $\dum(y_j,y_l)\ge r/2$. Since by condition~(i) we know that
$\overline{\Omega}^M$ is proper, we must necessarily have that $\{y_j\}$ has no bounded subsequence
with respect to the Mazurkiewicz metric $\dum$. By the sequential compactness of $\overline{\Omega}^P$
there must then be a prime end $\ms{E}$, and a subsequence, also denoted $\{y_j\}$, such that
this subsequence converges to $\ms{E}$ in the topology of $\overline{\Omega}^P$. Since $\{y_j\}$ is 
a bounded sequence in $X$, it follows from the properness of $X$ that there is a subsequence converging
to some $w\in\overline{\Omega}$. As $w\not\in\Omega$, we must have that $w\in\partial\Omega$,
and moreover, $w\in I(\ms{E})$. Thus $\ms{E}$ is of type~(a) since by condition~(ii) we have no
type~(b) prime ends. It then follows that when $\{E_k\}\in\ms{E}$, for sufficiently large $k$ the open connected
set $E_k$ must be bounded in $d_X$ and hence in $\dum$. This creates a conflict between 
the fact that the tail end of the sequence $\{y_j\}$ lies in $E_k$ and the fact that $\{y_j\}$ has no
bounded subsequence with respect to the metric $\dum$. Hence $\Omega$ must be finitely connected at $x_0$.
\end{proof}

\begin{ex}
Let $\Omega\subset\R^2$ be given by
\[
\Omega=(0,\infty)\times(0,1)\setminus\bigcup_{k\in\N} [0,k]\times\{2^{-k}\}.
\]
Then $\Omega$ has no type~(b) prime end. Moreover, any sequence $\{x_k\}$ in $\Omega$ that is
bounded in the Mazurkiewicz metric $\dum$ will have to lie in a subset $(0,\infty)\times(2^{-k},1)$
for some $k\in\N$, and so $\overline{\Omega}^M$ is proper. However, $\Omega$ is \emph{not} finitely
connected at the boundary. Thus Conditions~(i),(ii) on their own do not characterize finite connectedness at the boundary.
\end{ex}

\begin{lemma}\label{lem:bdd-sets-to-bdd-sets}
Suppose that $\overline{\Omega}^M$ proper. Let $A,B\subset\Omega$ with
$\dM(\overline{A}^M,\overline{B}^M)=\tau>0$ and $\diam_M(A), \diam_M(B)$ both finite. 
Then either $\diam_M^\prime(f(A))$ is finite or else $\diam_M^\prime(f(B))$ is finite.
\end{lemma}

\begin{proof}
Suppose that $\diam_M^\prime(f(A))=\infty=\diam_M^\prime(f(B))$. Then we can find two sequences 
$z_k\in f(A)$ and $w_k\in f(B)$ such that $\dumm(z_k,z_{k+1})\ge k$ and $\dumm(w_k,w_{k+1})\ge k$
for each $k\in\N$. Let $x_k=f^{-1}(z_k)$ and $y_k=f^{-1}(w_k)$. By the properness of $\overline{\Omega}^M$,
we can find two subsequences, also denoted $x_k$ and $y_k$, such that these subsequences are Cauchy
in $\dum$. We can also ensure that $\dum(x_k,x_{k+1})<2^{-k}\tau/8$ and $\dum(y_k,y_{k+1})<2^{-k}\tau/8$.
Then for each $k$ there is are continua $\alpha_k, \beta_k\subset\Omega$ with $x_k,x_{k+1}\in\alpha_k$,
$y_k,y_{k+1}\in\beta_k$, and $\max\{\diam_M(\alpha_k),\diam_M(\beta_k)\}<2^{-k}\tau/8$. Since 
$x_k\in A$, $y_k\in B$, and $\dM(\overline{A}^M,\overline{B}^M)=\tau>0$, we must have that 
$\bigcup_k\alpha_k$ and $\bigcup_k\beta_k$ are disjoint with
\[
\dM\left(\bigcup_k\alpha_k, \bigcup_k\beta_k\right)\ge 3\tau/4.
\]
For each $n\in\N$ set $\Gamma_n=\bigcup_{k=1}^n\alpha_k$ and $\Lambda_n=\bigcup_{k=1}^n\beta_k$.
Then $\Gamma_n$ and $\Lambda_n$ are disjoint continua with $\dM(\Lambda_n,\Gamma_n)\ge 3\tau/4$
and 
\[
\Gamma_n\cup\Lambda_n\subset \bigcup_{z\in A\cup B}B(z,\tau+\diam(A)+\diam(B))=:U. 
\]
Hence
\[
\Mod_Q(\Gamma(\Gamma_n,\Lambda_n))\le \int_\Omega\rho^Q\, d\mu<\infty
\]
where $\rho:=\tfrac{4}{3\tau}\chi_U$ is necessarily admissible for the family $\Gamma(\Gamma_n,\Lambda_n)$
of curves in $\Omega$ connecting points in $\Gamma_n$ to points in $\Lambda_n$.

However, $f(\Lambda_n)$ is a continuum containing $z_1$ and $z_n$, and so 
$\diam_M^\prime(f(\Lambda_n))\to\infty$ as $n\to\infty$; a similar statement holds for $f(\Gamma_n)$.
Moreover, $\dM^\prime(f(\Lambda_n),f(\Gamma_n))\le \dumm(z_1,w_1)<\infty$. Hence by the 
$Q$-Loewner property together with the Ahlfors $Q$-regularity of $(\Omega^\prime,\dumm)$ gives us
\[
\Mod_Q(\Gamma(f(\Gamma_n),f(\Lambda_n)))
  \ge \Phi\left(\frac{\dM^\prime(f(\Gamma_n),f(\Lambda_n))}
    {\min\{\diam_M^\prime(f(\Lambda_n)),\diam_M^\prime(f(\Gamma_n))\}}\right)\to\infty\text{ as }n\to\infty.
\]
this violates the geometric quasiconformality of $f$.
\end{proof}

\begin{lemma}\label{lem:bdd-geom-ends}
Suppose that $(\Omega,\dum)$ is Ahlfors $Q$-regular for some $Q>1$ and supports a $Q$-Poincar\'e inequality.
Then every type~(a) prime end of $\Omega$ is a singleton prime end, and $\overline{\Omega}^M$ is proper.
Moreover, $\overline{\Omega}^P$ is sequentially compact.
\end{lemma}

Thanks to Lemma~\ref{lem:finite-conn-at-bdy-implication}, if the domain $\Omega$ satisfies the hypotheses of the 
above lemma and in addition has no type~(b) ends, then $\Omega$ is finitely connected at the boundary.
As Example~\ref{ex:not-BQS-Plane-Sphere} shows, there are domains $\Omega$ for which 
$\overline{\Omega}^M$ is 
proper and $\overline{\Omega}^P$ is sequentially compact but $\Omega$ is not finitely connected at its boundary.
We do not know whether this is still possible if we also require $(\Omega,\dum)$ to be Ahlfors $Q$-regular for 
some $Q>1$ and support a $Q$-Poincar\'e inequality.

\begin{proof}
The properness of $\overline{\Omega}^M$ is a consequence of the fact that $\mathcal{H}^Q$ (this Hausdorff
measure taken with respect to the Mazurkiewicz metric $\dum$) is doubling and that $\overline{\Omega}^M$ is
complete, see for example~\cite[Lemma~4.1.14]{HKSTbook}.

Let $\ms{E}$ be a prime end of type~(a) for $\Omega$, and let $x_0\in I(\ms{E})$. Let $\{E_k\}\in\ms{E}$ 
such that $E_1$ is bounded (in $d_X$, and hence by the local connectedness of $X$ and the fact that $E_1$ is
open, in the Mazurkiewicz metric $\dum$ as well). Let $\{x_j\}$ be a sequence in $E_1$ 
with $x_j\in E_j$ such that it converges (in the 
metric $d_X$) to $x_0$. Then as this sequence is bounded in $\dum$, it follows from the properness of
$\overline{\Omega}^M$ that there is a subsequence, also denoted $\{x_j\}$, such that this subsequence
converges in $\dum$ to a point $\zeta\in\partial_M\Omega$. For $n\in\N$ 
we claim that $B_M(\zeta,2^{-n})\cap\Omega$ is connected. Indeed, by the definition of Mazurkiewicz
closure, we can find a Cauchy sequence $x_k\in\Omega$ with $\dum(x_k,\zeta)\to 0$.
For $x\in B_M(\zeta,2^{-n})\cap\Omega$ we have $\dum(x,\zeta)<2^{-n}$, and so for sufficiently
large $k$ we also have that $\dum(x,x_k)<2^{-n}$. Thus we can find a continuum in $\Omega$
connecting $x$ to $x_k$ with diameter smaller than $2^{-n}$. Note that each point on this
continuum necessarily then lies in $B_M(\zeta,2^{-n})\cap\Omega$. Now if 
$y\in B_M(\zeta,2^{-n})\cap\Omega$, then we can connect both $x$ and $y$ to $x_k$ for sufficiently large $k$ and hence
obtain a continuum in $B_M(\zeta,2^{-n})\cap\Omega$ connecting $x$ to $y$. Therefore
$B_M(\zeta,2^{-n})\cap\Omega$ is a connected set.
Let $F_n=B_M(\zeta,2^{-n})\cap\Omega$ that contain $x_j$ for sufficiently large $j$. Then
$\{F_n\}$ is a chain for $\Omega$ with $\{x_0\}=I(\{F_n\})$, and so it forms a singleton prime end
$\ms{F}=[\{F_n\}]$. We claim now that $\ms{F}\, |\, \ms{E}$. For each $k\in\N$ we know that
$x_j\in E_{k+1}$ for all $j\ge k+1$; as $\dM(\Omega\cap\partial E_k,\Omega\cap\partial E_{k+1})>0$,
it follows that for large $j$ the connected set $F_j$ contains $x_m$ for all $m\ge j+1$ with 
\[
\diam_M(F_j)<\frac12 \dM(\Omega\cap\partial E_k,\Omega\cap\partial E_{k+1});
\]
hence $F_j\subset E_k$ for sufficiently large $j$. Now the primality of $\ms{E}$ tells us that
$\ms{E}=\ms{F}$, that is, $\ms{E}$ is a singleton prime end.

Now we show that $\overline{\Omega}^P$ is sequentially compact. If $\{x_j\}$ is a sequence in $\Omega$
that is bounded in the Mazurkiewicz metric $\dum$, then by the properness of $\overline{\Omega}^M$ we see
that there is a subsequence that converges to a singleton prime end or to a point in $\Omega$.
Similarly, if we have a sequence of singleton prime ends that is bounded with respect to $\dum$ (and recall that
singleton prime ends correspond to points in the Mazurkiewicz boundary $\partial_M\Omega$, see~\cite{BBS2}),
then we have a convergent subsequence. It now only remains to consider sequences $\{\zeta_j\}$ from
$\overline{\Omega}^P$ such that the sequence is either an unbounded (in $\dum$) sequence of points from
$\overline{\Omega}^M$ or is a sequence of prime ends that are of types~(b) and~(c). In this case, we fix
$x_0\in\Omega$ and construct a prime end $[\{F_k\}]$ as follows. With $n_k=k$ in Lemma~\ref{lem:end-b-is-prime},
we choose $F_1$ to be the unbounded component of $\Omega\setminus \overline{B_M(x_0,1)}$
that contain infinitely many points from $\{\zeta_j\}$ if this sequence consist of points in $\overline{\Omega}^M$,
or tail-end of the chains $\{E_k^j\}\in\zeta_j$ for infinitely many $j$. We then choose $F_2$ to be the
unbounded component of $\Omega\setminus \overline{B_M(x_0,2)}$ that contain infinitely many of the points
$\{\zeta_j\}$ that are in $F_1$ if $\{\zeta_j\}\subset\overline{\Omega}^M$, or
 tail-end of the chains $\{E_k^j\}\in\zeta_j$ for infinitely many $j$ that also gave a tail-end for $F_1$. 
 Proceeding inductively, we obtain a chain $\{F_k\}$; by Lemma~\ref{lem:end-b-is-prime} we know that
 $[\{F_k\}]$ is a prime end of $\Omega$, and by the above construction and by a Cantor diagonalization,
 we have a subsequence of $\{\zeta_j\}$ that converges to $[\{F_k\}]$.
\end{proof}

Chains in prime ends of type~(b) have corresponding separating sets $R_k$, but no control over the
Mazurkiewicz diameter of the 
boundaries of the acceptable sets that make up the chains. The following lemma and its corollary
give us a good representative chain in the prime end of type~(b) that gives us control over the
boundary of the acceptable sets as well as dispenses with the need for the separating sets $R_k$.

\begin{lemma}\label{lem:standard-rep-b}
Let $\ms{E}$ be an end of type~(b) of $\Omega$.  Let $\{E_k\}$ be a representative chain 
for $\ms{E}$ and let $R_k$ be as given in Definition \ref{def:type-ab-ends}.  
Then, there is a representative chain $\{F_k\}$ of $\ms{E}$ and compact sets $S_k \subset X$  with the following properties:
\begin{enumerate}
\item[(i)] $F_k$ is a component of $\Omega \setminus S_k$,
\item[(ii)] $\diam_M(S_k \cap \Omega) < \infty$,
\item[(iii)] $\dist_M(S_k \cap \Omega, S_{k+1} \cap \Omega) > 0$.
\end{enumerate}
Furthermore, there are corresponding separating sets $R_k^F$ from 
Definition~\ref{def:type-ab-ends} that have the properties that 
$\dist_M(R_k^F \cap \Omega, R_{k+1}^F \cap \Omega) > 0$ and $\diam_M(R_k^F \cap \Omega) < \infty$.  
\end{lemma}

\begin{proof}
Let 
\[
T_k = \overline{(R_k \cap \Omega \setminus E_{k+1}) \cap E_k}.
\]  
We see $E_{k+1}$ belongs to a single component of $\Omega \setminus T_k$ because 
$E_{k+1}$ is connected and $E_{k+1} \cap T_k = \emptyset$.  This follows as 
$E_{k+1}$ is open in $X$ and $(R_k \cap \Omega \setminus E_{k+1}) \cap E_k = \emptyset$.  
We also see that $T_k \cap \Omega \subset E_k$ because $R_k \cap \partial E_k = \emptyset$ 
and so $T_k\cap\Omega=(R_k\cap E_k)\setminus E_{k+1}$. 

By Lemma~\ref{lem:E1-E2-separation},
we observe that there are no points $x \in \partial \Omega\cap E_k$ and 
$y\in\Omega\cap\partial E_{k+1}$ that belong to the same component of $\Omega \setminus T_k$. 
%

Let $S_k = T_{2k}$ and let $F_k$ be the component of $\Omega \setminus S_k$ containing 
$E_{2k + 1}$.  The sets $S_k$ are compact as they are closed subsets of the sets $R_{2k}$.  
We first verify properties (i) - (iii) and then show that $\{F_k\}$ is a chain 
equivalent to $\{E_k\}$.  Property (i) follows immediately from the definition of $F_k$.  
Property (ii) follows as $S_k \subseteq R_{2k}$ and $\diam_M(R_{2k}) < \infty$.  For 
property (iii), let $x \in S_k \cap \Omega$ and $y \in S_{k+1} \cap \Omega$.  
Let 
$E \subset \Omega$ be a continuum containing $x$ and $y$. 
As $x \in S_k$, we have $x \notin E_{2k+1}$.  As $y \in S_{k+1}$, we have 
$y \in E_{2k + 1}$, so there must exist a point $z_1 \in E \cap \partial E_{2k+1}$.  
Similarly, $x \notin E_{2k+2}$ and $y \in E_{2k+2}$, so there is a point 
$z_2 \in E \cap \partial E_{2k+2}$.  Hence, 
\[
0 < \dist_M(\partial E_{2k+1} \cap \Omega, \partial E_{2k+2} \cap \Omega) \leq \dum(z_1, z_2). 
\]
and, as $x,y$ are arbitrary, we have $\dist_M(S_k \cap \Omega, S_{k+1} \cap \Omega) > 0$.  

It remains to show that $F_k$ is a chain of type~(b) that is equivalent to $\{E_k\}$.  
We first note that $F_{k+1} \subset E_{2k + 1} \subset F_k$, from which equivalence follows.  
Indeed, $E_{2k+1} \subset F_k$ follows immediately from the definition of $F_k$.  
To see that $F_{k+1} \subset E_{2k+1}$, suppose that there is a point $y\in F_{k+1}\setminus E_{2k+1}$.
Let $x\in E_{2k+3}\subset F_{k+1}$;
then by the connectedness of the open set $F_{k+1}$ together with local connectedness of $\Omega$,
there is a continuum $\gamma\subset F_{k+1}$ containing the points $x,y$. Then as $x\in E_{2k+1}$ and
$y\not\in E_{2k+1}$, $\gamma$ intersects $\Omega\cap\partial E_{2k+1}$. Moreover, as $x\in E_{2k+2}$
and $y\not\in E_{2k+2}$, we must also have that $\gamma$ intersects $\Omega\cap \partial E_{2k+2}$.
Then as $T_{2k+1}$ separates $\Omega\cap\partial E_{2k+1}$ from $\Omega\cap\partial E_{2k+2}$, it follows
that $\gamma$ intersects $T_{2k+1}$, contradicting the fact that $\gamma\subset F_{k+1}$.
%
%

Properties (a) and (d) in Definition \ref{def:type-ab-ends} follow immediately from the above observation.  Property (c) follows from (iii) as 
$\Omega\cap\partial F_k\subseteq S_k \cap \Omega$. 
For property (b), we use the sets $R_k^F = T_{2k+1}$.  Compactness of $R_k^F$ follows 
as before: $R_k^F$ is a closed subset of the compact set $R_{2k+1}$.  We see 
$\diam_M(R_k^F \cap \Omega) \leq \diam_M(R_{2k+1} \cap \Omega) < \infty$.  
The separation property of $R_k^F$ is similar to the separation property 
for the sets $T_k$: let $x \in \partial F_k$ and $y \in \partial F_{k+1}$ 
and suppose $\gamma$ is a path in $\Omega \setminus T_{2k+1}$ with 
$\gamma(0) = x$ and $\gamma(1) = y$.  Then, $y \in T_{2k+2}$ so 
$y \in E_{2k+2}$ and $x \in T_{2k}$ so $x \notin E_{2k+1}$.  Thus, 
$\gamma$ contains points in $\partial E_{2k+1}$ and $\partial E_{2k+2}$.  
Considering an appropriate subpath of $\gamma$ as before leads to a 
point in $\gamma \cap T_{2k+1}$, a contradiction.

The inequality $\dist_M(R_k^F \cap \Omega, R_{k+1}^F \cap \Omega) > 0$ 
follows as in the proof that $\dist_M(S_k \cap \Omega, S_{k+1} \cap \Omega) > 0$.  
The inequality $\diam_M(R_k^F \cap \Omega) < \infty$ follows as 
$R_k^F \cap \Omega \subseteq R_{2k+1}$.
\end{proof}

The following corollary is a direct consequence of the above lemma and its proof.

\begin{cor}\label{cor:Reduce-b}
If $\{E_k\}$ is a representative chain of a type~(b) prime end of $\Omega$, then
there is a sequence $\{F_k\}$ of unbounded open connected subsets of $\Omega$ such that
\begin{enumerate}
  \item[(a)] $F_{k+1}\subset F_k$ for each $k\in\N$,
  \item[(b)] $\diam_M(\Omega\cap\partial F_k)<\infty$ for each $k\in\N$,
  \item[(c)] $\dM(\Omega\cap\partial F_k,\Omega\cap\partial F_{k+1})>0$ for each $k\in\N$,
  \item[(d)] $\emptyset\ne\bigcap_k\overline{F_k}\subset\partial\Omega$,
 \end{enumerate} 
and in addition, for each $k\in\N$ we can find $j_k,i_k\in\N$ such that 
\[
 F_{j_k}\subset E_k\, \text{ and }\, E_{i_k}\subset F_k.
\]
Moreover, if $\{F_k\}$ is a sequence of unbounded open connected subsets of $\Omega$ 
that satisfy Conditions~(a)--(d) listed above, then $\{F_{2k}\}$ is a representative chain
of a type~(b) prime end of $\Omega$.
\end{cor}

\begin{lemma}\label{image separation inequality}
Suppose that both $(\Omega,\dum)$ and $(\Omega^\prime,\dumm)$ are Ahlfors $Q$-regular for some
$Q>1$ and that they both support a $Q$-Poincar\'e inequality with respect to the respective 
Ahlfors regular
Hausdorff measure. Let $\ms{E}$ be an end of any type of $\Omega$.  
Let $\{E_k\} \in \ms{E}$ be a chain with $\diam_M(\Omega\cap\partial E_k) < \infty$ 
for all $k$ and let $F_k = f(E_k)$.  Then,
\begin{equation}\label{eq:positive-dist-bdy}
\dM^\prime(\Omega^\prime\cap\partial F_k,\Omega^\prime\cap\partial F_{k+1})>0.
\end{equation}
\end{lemma}

Chains with $\diam_M(\Omega\cap\partial E_k) < \infty$ for all $k$ exist for all types of ends by Lemma \ref{lem:bdd-geom-ends} (for type (a)), Lemma \ref{lem:standard-rep-b} (for type (b)), and Lemma \ref{lem:standard-rep} (for type (c)).  Here $\dM^\prime$ is the distance between the two sets with respect to the
Mazurkiewicz metric $\dumm$. 

\begin{proof}

Suppose \eqref{eq:positive-dist-bdy} is not the case for some positive integer $k$. Then we can find two sequences
$w_j\in \Omega^\prime\cap\partial F_k$ and $z_j\in \Omega^\prime\cap\partial F_{k+1}$ such that
$\dumm(w_j,z_j)\to 0$ as $j\to\infty$. Let $x_j\in\Omega\cap\partial E_k$ be such that $f(x_j)=w_j$,
and $y_j\in\Omega\cap\partial E_{k+1}$ be such that $f(y_j)=z_j$. 
By passing to a subsequence if necessary, we may assume that $w_1\ne w_2$ and that $z_1\ne z_2$;
therefore we also have $x_1\ne x_2$ and $w_1\ne w_2$.
Note that as
\[
\tau_k:=\dM(\Omega\cap\partial E_k,\Omega\cap\partial E_{k+1})>0,
\]
we must have that $\infty>\diam_M(E_1)\ge\dum(x_j,y_j)\ge \tau_k$ for each $j$.

By Lemma~\ref{lem:bdd-geom-ends} we know that $\overline{\Omega}^M$ is proper, and so sequences
that are bounded in $\Omega$ have a subsequence that converges in $\dum$ (and hence in 
$d_X$ as well) to some point in $\overline{\Omega}^M$. By~\cite[Theorem~1.1]{BBS1}
we know that there is a bijective correspondence between singleton prime ends and
points in $\partial_M\Omega$.
Therefore
by passing to a subsequence if necessary, we can find $\zeta,\xi\in\partial_P\Omega$ with 
$\zeta\ne \xi$
such that $x_j\to\zeta$ and $y_j\to\xi$. Note that $\dum(\xi,\zeta)\ge \tau_k$. So we can choose
chains $\{G_m\}\in\xi$ and $\{H_m\}\in\zeta$ such that for each $m$ the completion in $\dum$ of 
$G_m$ and the completion of $H_m$ in $\dum$ are disjoint (indeed, we can ensure that
$G_m\subset B_M(\xi, \tau_k/3)$ and $H_m\subset B_M(\zeta,\tau_k/3)$). As $x_j\to\zeta$, 
by passing to a further subsequence if needed, we can ensure that for each positive
integer $j$, $x_j\in H_1$; similarly, we can assume that $y_j\in G_1$. As $H_1$ and $G_1$ are open connected
subsets of $\Omega$ and $X$ is locally path-connected, it follows that for each $j$ there is a path
$\gamma_j$ in $H_1$ connecting $x_j$ to $x_{j+1}$; similarly there is a path $\beta_j$ in $G_1$
connecting $y_j$ to $y_{j+1}$. For each positive integer $n$ let $\alpha_n$ be the concatenation 
of $\gamma_j$, $j=1,\cdots, n$, and let $\sigma_n$ be the concatenation of $\beta_j$,
$j=1,\cdots, n$. Then $\alpha_n$ and $\sigma_n$ are continua in $\Omega$ with 
$\dM(\alpha_n,\sigma_n)\ge \tau_k/3$. Therefore for each $n$ we have
\[
\Mod_Q(\Gamma(\alpha_n,\sigma_n))\le \frac{3^Q}{\tau_k^Q}\mu(\Omega)<\infty.
\]
On the other hand, with $A_n=f(\alpha_n)$ and $\Sigma_n=f(\sigma_n)$, we know that
$\dM^\prime(A_n,\Sigma_n)\to 0$ as $n\to\infty$. 
However, 
\begin{align*}
\liminf_{n\to\infty}\diam_M^\prime(A_n)&\ge \diam_M^\prime(A_1)>0,\\
\liminf_{n\to\infty}\diam_M^\prime(\Sigma_n)&\ge \diam_M^\prime(\Sigma_1)>0.
\end{align*}
Therefore by the Ahlfors regularity together with
the Poincar\'e inequality on $\Omega^\prime$ with respect to $\dumm$, we know that
\[
\Mod_Q(\Gamma(A_n,\Sigma_n))\to\infty\text{ as }n\to\infty.
\]
Since
\[
f(\Gamma(\alpha_n,\sigma_n))=\Gamma(A_n,\Sigma_n),
\]
this violates the geometric quasiconformality of $f$, see Theorem~\ref{thm:HeiK}. 
Hence~\eqref{eq:positive-dist-bdy} must hold true.
\end{proof}

\begin{lemma}\label{lem: ends-to-ends}
Suppose that both $(\Omega,\dum)$ and $(\Omega^\prime,\dumm)$ are Ahlfors $Q$-regular for some
$Q>1$ and that they both support a $Q$-Poincar\'e inequality with respect to the respective 
Ahlfors regular
Hausdorff measure. Let $\ms{E}$ be an end of any type of $\Omega$.  
Let $\{E_k\} \in \ms{E}$ be an appropriate chain with 
$\diam_M(\Omega\cap\partial E_k)<\infty$ for all $k$ and let $F_k = f(E_k)$.  Then,
\begin{enumerate}
\item[(i)] If $\diam(F_k) < \infty$ for some $k$, then $\{F_k\}$ is a chain 
corresponding to an end of type (a).
\item[(ii)] If $\diam(F_k) = \infty$ for all $k$ and $\bigcap_k \overline{F_k} \neq \emptyset$, 
then $\{F_k\}$ is a chain corresponding to an end of type (b).
\item[(iii)] If $\diam(F_k) = \infty$ for all $k$ and $\bigcap_k \overline{F_k} = \emptyset$, 
then $\{F_k\}$ is a chain corresponding to an end of type (c).
\end{enumerate}
\end{lemma}

We use the phrase ``appropriate chain'' to mean one which has the properties in Lemma \ref{lem:standard-rep-b} for type (b) or of the form given in Lemma \ref{lem:standard-rep} for type (c).  As $f$ is a homeomorphism, this means that $f$ induces a map from ends of $\Omega$ to ends of $\Omega'$.  

\begin{proof}
Let $\ms{E}$ be an end and let $\{E_k\}$ be a chain representing $\ms{E}$ with 
$\diam_M(\Omega\cap\partial E_k) < \infty$ for all $k$.  Let $F_k = f(E_k)$.  
The proof will be split into several cases.  We note that in all cases, the containments $F_{k+1} \subseteq F_k$ are immediate and the separation inequality~\eqref{eq:positive-dist-bdy} follows 
from Lemma~\ref{image separation inequality}, so we only check the other conditions.

\noindent {\bf Case 1 [(a),(b),(c) $\to$ (a)]:} Suppose $\diam(F_k) < \infty$ for some $k$. By considering the tail end of this sequence, we assume $\diam(F_k) < \infty$ for all $k$.  Then, as $f$ is a homeomorphism, each $F_k$ is a bounded, connected, open set.  The sets $\overline{F_k}$ are compact as $Y$ is proper.  As $Y$ is complete, it follows that $\bigcap_k \overline{F_k} \neq \emptyset$.  The fact that $\bigcap_k \overline{F_k} \cap \Omega' = \emptyset$ follows as $f$ is a homeomorphism: if $x' = f(x) \in \bigcap_k \overline{F_k} \cap \Omega'$, then as $\Omega'$ is open we can find $r'>0$ with $B(x', r) \subset \Omega'$.  Then, there is an $r > 0$ with $B(x, r) \subset f^{-1}(B(x', r'))$.  For all types of ends, there is an index $I$ with $E_I \cap B(x, r/2) = \emptyset$.  Hence $f(B(x, r/2)) \cap F_I = \emptyset$, and so $x \notin \overline{F_I}$.  In particular, $\overline{F_k} \cap \partial \Omega' \neq \emptyset$ for all $k$, so the sets $F_k$ are admissible, and $\bigcap_k \overline{F_k} \subset \partial \Omega$, so $\{F_k\}$ is a chain.

\noindent {\bf Case 2 [(a),(b),(c) $\to$ (b)]:} We assume that $\diam(F_k) = \infty$ for all $k$ and that $\bigcap_k \overline{F_k} \neq \emptyset$.  As in Case 1, we must have $\bigcap_k \overline{F_k} \subseteq \partial \Omega'$.  Moreover, $\overline{F_k}$ is proper for all $k$ as $Y$ is proper.  
Now the fact that $\{F_k\}$ corresponds to an end of type~(b) follows from Corollary~\ref{cor:Reduce-b}
together with Lemma~\ref{lem:bdd-sets-to-bdd-sets}.
%
%
%

\noindent {\bf Case 3 [(a),(b),(c) $\to$ (c)]:}  We assume that $\diam(F_k) = \infty$ 
for all $k$ and that $\bigcap_k \overline{F_k} = \emptyset$.  We only need to check 
that there exists a compact subset $K_k \subset Y$ with 
$\diam_M^\prime(K_k \cap \Omega') < \infty$ such that $F_k$ is a component of 
$\Omega' \setminus K_k$.  Note that for each $k$ we have
$\diam_M(\Omega\cap\partial E_k)<\infty$, and 
$\dM(\Omega\cap\partial E_k,\Omega\cap\partial E_{k+1})>0$. Therefore
for all except at most one $k$ we have that $\diam_M^\prime(\Omega^\prime\cap\partial f(E_k))<\infty$
by Lemma~\ref{lem:bdd-sets-to-bdd-sets}. Set $K_k=\overline{f(\partial E_k \cap \Omega')}$.
To see that 
$F_k=f(E_k)$ is a component of $\Omega' \setminus K_k$, we note that as $F_k$ is connected it 
is contained in some component of $\Omega' \setminus K_k$.  If $F_k$ is not this 
entire component, then as open connected sets are path connected, this component 
contains a point of $\partial F_k$, but this is impossible as $\partial F_k \subseteq K_k$ 
because $f$ is a homeomorphism.
%
%
%
\end{proof}

\begin{thm}\label{thm:main-geomQC}
Suppose that both $(\Omega,\dum)$ and $(\Omega^\prime,\dumm)$ are Ahlfors $Q$-regular for some
$Q>1$ and that they both support a $Q$-Poincar\'e inequality with respect to the respective Ahlfors regular 
Hausdorff measure.  Let $f \colon \Omega \to \Omega'$ be a geometric quasiconformal mapping.  Then, 
$f$ induces a homeomorphism $f \colon \overline{\Omega}^P \to \overline{\Omega'}^P$.
\end{thm}

\begin{proof}
By Lemma \ref{lem: ends-to-ends}, $f$ induces a map from ends of $\Omega$ to ends of $\Omega'$.  As 
continuity is determined by inclusions of sets in chains and $f$ as a homeomorphism preserves these 
inclusions, the extended $f$ is also a homeomorphism.   

To see that $f$ maps prime ends to prime ends, suppose that $\ms{E}$ is a prime end of $\Omega$.  Then, 
by Lemma~\ref{lem: ends-to-ends}, $f(\ms{E})$ is an end of $\Omega'$.  
If $\ms{F}$ is an end with $\ms{F} | f(\ms{E})$, then by applying Lemma~\ref{lem: ends-to-ends} to $f^{-1}$, 
we see that $f^{-1}(\ms{F})$ is an end with $f^{-1}(\ms{F}) | \ms{E}$.  As $\ms{E}$ is prime, it follows 
that $\ms{E} | f^{-1}(\ms{F})$ as well, and so $f(\ms{E}) | \ms{F}$.  It follows that $f(\ms{E})=\ms{F}$, 
so $f(\ms{E})$ is prime as $\ms{F}$ was arbitrary.
\end{proof}

We end this paper by stating the following two open problems:
\begin{enumerate}
  \item If $\Omega$ is a domain (open connected set) in a proper locally path-connected metric space and
$(\Omega,\dum)$ is Ahlfors $Q$-regular and supports a $Q$-Poincar\'e inequality for some $Q>1$, then could
$\Omega$ have a prime end of type~(b)?
  \item If $\Omega$ and $\Omega^\prime$ are domains that satisfy the hypotheses of this section and 
$f:\Omega\to\Omega^\prime$ is geometrically quasiconformal, then is its homeomorphic extension
$f:\overline{\Omega}^P\to\overline{\Omega^\prime}^P$ also quasiconformal, and does quasiconformality
make sense in this situation when $\overline{\Omega}^P$ and/or $\overline{\Omega^\prime}^P$ are not
metric spaces but are only topological spaces? If $\Omega$ and $\Omega^\prime$ are Euclidean domains
with $\overline{\Omega}^P=\overline{\Omega}$ and $\overline{\Omega^\prime}^P=\overline{\Omega^\prime}$
then the results of~\cite{JoSmr} give a criterion under which $f$ extends to a quasiconformal mapping.
\end{enumerate}

\vskip .5cm

\noindent Address:\\

\vskip .2cm

Department of Mathematical Sciences, University of Cincinnati, P.O.Box~210025, Cincinnati, OH 45221-0025, USA.\\

\noindent E-mail: J.K.:~{\tt klinejp@mail.uc.edu}, J.L.:~{\tt lindqujy@uc.edu}, N.S.:~{\tt shanmun@uc.edu}

\end{document}